\documentclass[a4paper, 12 pt, oneside, reqno]{amsart} 
\usepackage{amsfonts, amssymb, amsmath, eucal, amsthm}
\usepackage[all]{xy}
\usepackage{verbatim}
\usepackage{tikz}
\usepackage[hidelinks]{hyperref}
\usepackage{bbm}
\usepackage{color}
\usepackage{mathtools}

\newtheorem{lemma}{Lemma}
\numberwithin{lemma}{section}
\newtheorem{theorem}[lemma]{Theorem}
\newtheorem*{theorem*}{Theorem}
\newtheorem*{corollary*}{Corollary}
\newtheorem{proposition}[lemma]{Proposition}

\theoremstyle{definition}
\newtheorem{definition}[lemma]{Definition}
\newtheorem{note}[lemma]{Note}
\newtheorem{example}[lemma]{Example}
\newtheorem{remark}[lemma]{Remark}

\newtheorem*{acknowledgment}{Acknowledgement}
\newtheorem*{remark*}{Remark}

\usepackage{geometry}

\renewcommand{\H}{\mathcal{H}}
\newcommand{\C}{\mathcal{C}}
\newcommand{\D}{\mathcal{D}}
\newcommand{\bbC}{\mathbb{C}}
\newcommand{\bbF}{\mathbb{F}}
\newcommand{\A}{\mathcal{A}}
\renewcommand{\S}{\mathfrak{S}}
\renewcommand{\t}{\mathbbm{1}}

\DeclareMathOperator{\tor}{Tor}
\DeclareMathOperator{\ext}{Ext}
\let\hom\relax
\DeclareMathOperator{\hom}{Hom}
\DeclareMathOperator{\colim}{colim}
\newcommand{\IE}{{}^{I}E}
\newcommand{\IIE}{{}^{II}E}

\newcommand{\bfs}{\mathbf{s}}

\title{Homological stability for Iwahori-Hecke algebras}
\author{Richard Hepworth}
\address{Institute of Mathematics\\
University of Aberdeen
}
\email{r.hepworth@abdn.ac.uk}



\begin{document}

\begin{abstract}
	We show that the Iwahori-Hecke algebras $\H_n$
	of type $A_{n-1}$ satisfy homological stability,
	where homology is interpreted as an appropriate Tor group.
	Our result precisely recovers Nakaoka's homological stability result
	for the symmetric groups in the case that the 
	defining parameter is equal to $1$.
	We believe that this paper, and our joint work with Boyd
	on Temperley-Lieb algebras, are the first time that the techniques
	of homological stability have been applied to algebras that 
	are not group algebras.
\end{abstract}

\maketitle

\section{Introduction}

\subsection{Homological stability}

A family of discrete groups 
\[
	G_0\hookrightarrow G_1\hookrightarrow G_2\hookrightarrow\cdots
\]
satisfies \emph{homological stability} if the maps
\[
	H_d(G_{n-1})\longrightarrow H_d(G_n)
\]
are isomorphisms when $n$ is sufficiently large compared to $d$.  
Homological stability can similarly be formulated for sequences
of topological groups, and for families of spaces that are not necessarily
classifying spaces of groups.
Examples of families for which homological stability holds
include symmetric groups~\cite{Nakaoka},
general linear groups~\cite{Quillen, Charney, vanderKallen},
mapping class groups of surfaces and 
3-manifolds~\cite{Harer, RandalWilliamsMCG, WahlMCG, HatcherWahl},
diffeomorphism groups of highly connected 
manifolds~\cite{GalatiusRandalWilliams}, automorphism groups of 
free groups~\cite{HatcherVogtmannStability,HatcherVogtmannRational},
families of Coxeter groups~\cite{HepworthCoxeter} and Artin monoids~\cite{Boyd},
configuration spaces of manifolds~\cite{Church}, \cite{RandalWilliamsConfig},
and a great many others besides.

The homology $H_\ast(G;R)$ 
of a discrete group $G$ with coefficients in a ring $R$ 
can be written as the $\tor$ group
\[
	\tor^{RG}_\ast(\t,\t)
\]
over the group algebra $RG$, where $\t$ denotes the trivial representation.
This formulation shows that the homology of a group depends only on the 
group algebra.
We can therefore say that a family of algebras
\[
	\A_0\to \A_1\to \A_2\to\cdots
\]
equipped with a consistent choice of `trivial representation' $\t$
satisfies \emph{homological stability} if the maps
\[
	\tor^{\A_{n-1}}_d(\t,\t)
	\longrightarrow	
	\tor^{\A_{n}}_d(\t,\t)
\]
are isomorphisms when $n$ is sufficiently large compared to $d$.
Here the algebras need not be group algebras, and the only requirement
on $\t$ is that it is a module for each $\A_n$
and that the module structures are compatible with the  
maps $\A_{n-1}\to \A_n$.

The purpose of this paper is to demonstrate that homological stability
holds in this sense for Iwahori-Hecke algebras of type $A_{n-1}$, 
and moreover that it can be proved by adapting the suite of techniques
used to study families of groups to the setting of algebras.
In~\cite{BoydHepworthStability}, Boyd and the author prove 
homological stability for the \emph{Temperley-Lieb algebras}.
There we again use the techniques of homological stability, 
but encounter --- and resolve --- novel obstructions 
that are not present in the setting of groups or of Iwahori-Hecke algebras.

To the best of our knowledge, the present paper 
and~\cite{BoydHepworthStability} are the first homological stability
results of their kind for algebras that are not group algebras,
and our hope is that they will serve as a proof of concept for the export
of homological stability techniques into new algebraic contexts.
Indeed, since the appearance of the present paper 
and~\cite{BoydHepworthStability}, 
the author together with Boyd and Patzt have used the same set of techniques 
to study the homology of Brauer algebras~\cite{BoydHepworthPatzt}.
On a related note, Sroka~\cite{Sroka} has adapted techniques from the
geometry and topology of Coxeter groups 
(specifically the Davis complex) to study
the homology of odd Temperley-Lieb algebras.

\subsection{Iwahori-Hecke algebras}

The symmetric group $\S_n$ has presentation with generators 
\[
	s_1,\ldots,s_{n-1},
\]
and with relations
\begin{align*}
	s_{i}s_j &= s_js_i  && \text{for }|i-j|>1,
	\\
	s_{i}s_js_i &= s_js_is_j  && \text{for }|i-j|=1,
	\\
	s_i^2&= e && \text{for all }i.
\end{align*}
where $s_i$ is the adjacent transposition $s_i=(i\ i+1)$.
This is the presentation of $\S_n$ as 
the \emph{Coxeter group of type $A_{n-1}$}.

Now let $R$ be a commutative ring and let $q\in R^\times$ be a unit.
The \emph{Iwahori-Hecke algebra of type $A_{n-1}$}, 
denoted $\H_n$, is the $R$-algebra with generators
\[
	T_1,\ldots,T_{n-1}
\]
and with relations
\begin{align*}
	T_{i}T_j &= T_jT_i  && \text{for }|i-j|>1,
	\\
	T_{i}T_jT_i &= T_jT_iT_j  && \text{for }|i-j|=1,
	\\
	(T_i+1)(&T_i-q)= 0 && \text{for all }i.
\end{align*}
When $q=1$, the final relation can be rewritten as $T_i^2=1$,
so that $\H_n\cong R\S_n$ by the isomorphism that sends $T_i$ to $s_i$.
Thus $\H_n$ is a `deformation' of $R\S_n$ depending on the parameter $q$.  
Taking $R=\bbC$, then $\H_n\cong\bbC\S_n$ unless 
$q$ is a $d$-th root of unity 
for $2\leqslant d\leqslant n$~\cite[Theorem~2.2]{Wenzl},
in which case no such isomorphism exists.

The algebras $\H_n$ are important from several points of
view, and we mention just a couple.
In knot theory, the $\H_n$ are a crucial ingredient in certain definitions of
the {\sc homfly-pt} polynomial~\cite{HOMFLY,Jones},
and their categorifications via Soergel bimodules are 
used to define categorifications of this polynomial~\cite{KhovanovTriply}.
In representation theory, if we take $R=\bbC$ and $q$ a prime power, 
then $\H_n$ is isomorphic to the endomorphism algebra of a certain 
representation of $\mathrm{GL}_n(\mathbb{F}_q)$,
and this allows the construction of an irreducible representation of
$\mathrm{GL}_n(\mathbb{F}_q)$ from each irreducible of $\S_n$,
see~\cite[pp.x-xi]{Mathas}.
For general introductions to the $\H_n$ we 
suggest~\cite[Chapter 1]{Mathas} and~\cite[Chapters 4-5]{KasselTuraev}.

In general, there is an Iwahori-Hecke algebra associated to any
Coxeter system,
and these more general Iwahori-Hecke algebras are important in many parts
of representation theory, see for example~\cite{GeckPfeiffer},
\cite[Chapter 7]{HumphreysReflection},
\cite{KazhdanLusztig} and \cite{Libedinsky}.
We will often refer to the $\H_n$ as \emph{Iwahori-Hecke algebras}
without explicitly mentioning their type.

\subsection{Homological stability for Iwahori-Hecke algebras}
The Iwahori-Hecke algebra $\H_n$ has two natural rank-1 modules,
denoted $\t$ and $\varepsilon$, where each $T_i$ acts on $\t$
as multiplication by $q$, and on $\varepsilon$ as multiplication by $(-1)$,
see Corollary~1.14 of~\cite{Mathas}.
When $q=1$ the modules $\t$ and $\varepsilon$ become the trivial representation
and the sign representation respectively.  
We may therefore consider 
\[
	\tor^{\H_n}_\ast(\t,\t)
	\quad\text{and}\quad
	\ext_{\H_n}^\ast(\t,\t)
\]
to be the \emph{homology} and \emph{cohomology} of $\H_n$,
and indeed when $q=1$ these become simply $H_\ast(\S_n;R)$
and $H^\ast(\S_n;R)$ respectively.
We can now state our main result:

\begin{theorem}\label{theorem-main}
	The maps
	\[
		\tor^{\H_{n-1}}_d(\t,\t)
		\longrightarrow
		\tor^{\H_n}_d(\t,\t)
	\]
	and
	\[
		\ext_{\H_n}^d(\t,\t)
		\longrightarrow
		\ext_{\H_{n-1}}^d(\t,\t)
	\]
	are isomorphisms for $d\leqslant\frac{n-1}{2}$.
\end{theorem}

When $q=1$ then $\H_n\cong R\S_n$, and Theorem~\ref{theorem-main} 
gives exactly Nakaoka's stability result 
for the homology and cohomology of symmetric groups.
See \cite[Corollary 6.7]{Nakaoka}, \cite[Theorem~2]{Kerz}
and \cite[Theorem~5.1]{RandalWilliamsConfig}.
Nakaoka in fact gave a complete computation
of $H_\ast(\S_n;\bbF_p)$ for any prime $p$, and this can be used to show that
for $k\geqslant 1$ the map $H_k(\S_{2k-1};\bbF_2)\to H_k(\S_{2k};\bbF_2)$ 
is not surjective.  Thus the range $d\leqslant\frac{n-1}{2}$ 
appearing in the theorem cannot be improved in general.

\subsection{Comparison with work of Benson-Erdmann-Mikaelian}
The cohomology ring $\ext_{\H_n}^\ast(\t,\t)$ of $\H_n$ was 
explicitly computed by
Benson, Erdmann and Mikaelian~\cite{BEM} in the case where 
$R=\bbC$ and $q$ is a primitive $\ell$-th root of unity with $\ell\geqslant 2$.
In the case $\ell>n$, the result of Wenzl mentioned above shows that
$\H_n\cong\bbC\S_n$, so that $\ext_{\H_n}^d(\t,\t)=H^\ast(\S_n;\bbC)$ 
is trivial,
but when $2\leqslant \ell\leqslant n$ then no such isomorphism holds,
and indeed Benson-Erdmann-Mikaelian
show that $\ext_{\H_n}^d(\t,\t)$ is nontrivial.
Furthermore, one can use their results to observe that in this case 
the stabilisation maps $\ext_{\H_n}^d(\t,\t)\to\ext_{\H_{n-1}}^d(\t,\t)$ 
are isomorphisms up to and including (at least) degree $(n-2)$.
So \cite{BEM} serves as an antecedent of the present work,
but more interestingly, it demonstrates a much stronger
stable range in this case, of slope $1$ rather than slope $\frac{1}{2}$.  
This is reminiscent of the slope $1$ rational
homological stability results for configuration spaces of manifolds
(see for example Corollary~3 of~\cite{Church} 
and Theorem~B of~\cite{RandalWilliamsConfig}).
It suggests that there may be a slope $1$ stability
result for the $\H_n$ in characteristic $0$.

\subsection{Discussion:
Homological stability for Coxeter groups and Artin monoids}
The present paper builds strongly on previous work of the 
author~\cite{HepworthCoxeter}, which proved homological stability
for families of Coxeter groups, and of Boyd~\cite{Boyd}, which
proved homological stability for families of Artin monoids.  
These papers
demonstrated that one can do all of the normal
work of a homological stability proof purely in terms of a Coxeter or Artin-type
presentation, rather than in terms of a concrete model of the group
or monoid being studied.
The defining presentation of the Iwahori-Hecke algebra $\H_n$
is of course very close to both of these, being a deformation of the
Coxeter presentation of $\S_n$, and a quotient of the Artin presentation
of the braid group (or rather of their group rings).

In both \cite{HepworthCoxeter} and \cite{Boyd},
the results apply to families of groups or monoids obtained from
sequences of Coxeter diagrams 
that `grow a tail' of type $A_{n-1}$ as $n$ increases.
These families are very general,
but include as the basic case the families of type $A$,
$B$ and $D$.
So one may ask whether Theorem~\ref{theorem-main}
can be extended to apply to any of these more general families.
This seems likely, but we were not able to prove Theorem~\ref{theorem-main}
by generalising the method of~\cite{HepworthCoxeter}
from Coxeter groups to Iwahori-Hecke algebras; this is discussed
further in section~\ref{subsection-method} below.

\subsection{Discussion: Stable homology}
Theorem~\ref{theorem-main} shows that, in a fixed degree $d$,
then for $n$ sufficiently large the groups $\tor_d^{\H_n}(\t,\t)$
all agree and coincide with the \emph{stable homology}
\[
	\colim_n\tor^{\H_n}_d(\t,\t) = \tor^{\H_\infty}_d(\t,\t),
\]
where $\H_\infty = \colim_n\H_n$ is the `infinite' Iwahori-Hecke algebra.
When $q=1$, the stable homology $\tor^{\H_\infty}_\ast(\t,\t)$ coincides
with the homology of the infinite symmetric group, $H_\ast(\Sigma_\infty;R)$,
which is computed by the Barratt-Priddy-Quillen theorem
\cite{BarrattPriddy}, \cite{QuillenQ}:
\[
	H_\ast(\Sigma_\infty;R)\cong H_\ast(\Omega_0^\infty S^\infty;R).
\]
Here $\Omega^\infty S^\infty=\colim_n\Omega^nS^n$ 
is the infinite loop space of the sphere spectrum,
and $\Omega_0^\infty S^\infty$ is the path component of its basepoint.
It is therefore natural to ask what is the stable homology
$\tor^{\H_n}_\ast(\t,\t)$ in general?
To put it another way, 
what is the Iwahori-Hecke analogue of $H_\ast(\Omega_0^\infty S^\infty;R)$?

\subsection{Discussion: Homological stability for algebras}

As we said earlier, we believe that the work of the present paper 
on Iwahori-Hecke algebras and of~\cite{BoydHepworthStability} on
Temperley-Lieb algebras
are the first time the techniques of homological stability have
been applied to families of algebras that are not group algebras,
and we hope that they will serve as a starting point for new work
in this area.
We refer the reader to the
introduction of~\cite{BoydHepworthStability}, where several 
possible directions are discussed in some detail.

\subsection{Method of proof}\label{subsection-method}

Proofs of homological stability for sequences of groups $(G_n)_{n\geqslant 0}$
can often be placed in the following broad framework:
\begin{itemize}
	\item
	Find a complex (a simplicial complex, or semisimplicial
	set, or chain complex) upon which the $n$-th group $G_n$
	acts in such a way that the
	simplex stabilisers are of the form $G_m$ for $m<n$.
	(Or at least, the simplex stabilisers must be 
	associated to the previous groups in the sequence in some way).
	\item
	Prove that the complex is highly acyclic, i.e.~that its homology
	vanishes up to a certain point.
	\item
	Use an algebraic method (often but not always a spectral
	sequence argument) based on the complex in order 
	to prove stability by induction.
\end{itemize}
While many different proofs fit this framework when viewed from a distance,
there are many choices to be made and many variations are possible.
It may be possible to prove stability for the same family of groups
by choosing different complexes to begin with.  It may be possible
to prove high-acyclicity of the same complex in multiple ways.  And it may
be possible to use the same complex in different algebraic arguments
to prove stability.

Our approach to proving Theorem~\ref{theorem-main} fits into the framework
outlined above.  There is a well-known complex, called the 
\emph{complex of injective words}, that is used in many proofs of 
homological stability for the symmetric group.
For our complex, we construct an Iwahori-Hecke analogue of the complex
of injective words. 
While the complex of injective words has an action
of $\S_n$, our new complex is a chain complex of $\H_n$-modules;
and while the generators of the complex of injective words have
stabilisers given by smaller symmetric groups, our new complex is built
out of tensor products like $\H_n\otimes_{\H_{m}}\t$ for $m<n$.
The proof that our complex is highly acyclic is closely modelled on,
but far more involved than, a proof that the complex of injective
words is highly acyclic, and requires us to make careful use of
the theory of distinguished coset representatives in Coxeter
groups and the basis theorem for Iwahori-Hecke algebras.
Furthermore, new difficulties arise 
because $q$ is no longer equal to $1$, so that
one must now account for many hitherto-invisible powers of $q$.
(Surprisingly, the formula $T_k^2 = (q-1)T_k+q$ explicitly 
surfaces in only one place, and it quickly disappears again.)
The final step of our argument is a spectral sequence argument 
closely related to ones in the literature.

A lot of the difficulty in the present paper boils down to the fact
that we are operating under two significant constraints.
First, we are not working with the symmetric group, but with its
Iwahori-Hecke algebra, and while $\H_n$ is closely related to 
$\S_n$ \emph{thought of as a Coxeter group}, it is not useful to
think of $\H_n$ in terms of permutations of the set $\{1,\ldots,n\}$.
This means that we can approach the complex of injective words
only in terms of the Coxeter presentation of $\S_n$.
Second, the \emph{linear} nature of Iwahori-Hecke algebras
heavily restricts the suite of topological tools that we can apply.
For example, the approach of~\cite{HepworthCoxeter} to proving homological
stability for Coxeter groups could not be adapted to this setting
since it made use of simplicial complexes and barycentric subdivision,
which do not seem to have analogues in the linear setting.

There are by now several systematic approaches to proving homological stability
results, for example Randal-Williams and Wahl's 
approach~\cite{RandalWilliamsWahl} via homogeneous categories, 
the author's approach via families of groups with 
multiplication~\cite{HepworthEdge}, and 
Kupers, Galatius and Randal-Williams' approach via 
cellular $E_k$-algebras~\cite{GKRW-Ek}.  
One may ask whether the present
results could be proved using any of these frameworks.  
In the first two cases the answer is no, 
since these are designed purely for the study of groups,
though it is plausible that a `linearised' version of \cite{RandalWilliamsWahl}
would produce the same complex that we use.
In the final case, it seems that the methods of~\cite{GKRW-Ek} 
could possibly be applied in the present situation,
but we have taken a significantly more elementary approach.

\subsection{Outline of the paper}
\begin{itemize}
	\item
	We begin in section~\ref{section-iwahori-hecke} with some detailed
	background on Iwahori-Hecke algebras.
	\item
	In section~\ref{section-injective-words} 
	we give a short account of the complex of injective
	words and a proof that it is highly-acyclic.
	\item
	In section~\ref{section-translation} we rephrase the complex
	of injective words entirely in terms of the group ring $R\S_n$,
	where $\S_n$ is regarded as a Coxeter group.
	This formulation places the existing work on the symmetric groups
	in a setting where it can be extended to Iwahori-Hecke algebras.
	\item
	In section~\ref{section-Dn-overview} we give an overview of the
	construction of our analogue of the complex of injective words,
	$\D(n)$, and the proof of its high-acyclicity. 
	We set out how our approach builds on the rephrasing given in 
	section~\ref{section-translation},
	we describe the difficulties that arise,
	and we preview the work that follows in sections~\ref{section-Dn},
	\ref{section-filtration} and~\ref{section-filtration-quotients}.
	\item
	In section~\ref{section-Dn} we define $\D(n)$.
	\item
	In sections~\ref{section-filtration} 
	and \ref{section-filtration-quotients},
	we show that the homology of $\D(n)$ is zero up to 
	and including degree $(n-2)$.  
	Section~\ref{section-filtration} defines a filtration
	of $\D(n)$, while section~\ref{section-filtration-quotients}
	identifies the filtration quotients in terms of the 
	$\D(m)$ for $m<n$, allowing an inductive proof
	of high-acyclicity.
	\item
	In section~\ref{section-spectral-sequence} we obtain a spectral
	sequence from $\D(n)$ and identify its $E_1$ and $E_\infty$ terms.
	\item
	In section~\ref{section-argument} we use the spectral sequence 
	to give an inductive proof of Theorem~\ref{theorem-main}.
\end{itemize}

\begin{acknowledgment}
	I would like to thank the anonymous referee, whose comments helped to 
	significantly improve the readability of the paper.
\end{acknowledgment}

\section{Background on Iwahori-Hecke algebras}\label{section-iwahori-hecke}

This section is a rapid run through the theory of Coxeter groups
and Iwahori-Hecke algebras that is necessary for the applications
in this paper.  The intention is to give the reader a flavour of the
extent and depth of the theory required.  
Everything that we recall here is basic in the theory of Coxeter groups 
and Iwahori-Hecke algebras,
but it nevertheless amounts to a significant amount of nontrivial theory.
However, none of this theory
is strictly necessary until subsection~\ref{subsection-isomorphism}, and
some readers may wish to skim or skip the section until then.
For further reading we recommend 
chapters~1, 2 and~4 of~\cite{GeckPfeiffer},
chapters~3 and~4 of~\cite{Davis},
or chapter~1 of~\cite{Mathas} in the $\H_n$-case.

\subsection{Coxeter systems and Coxeter groups}

A \emph{Coxeter matrix} on a set $S$ is a symmetric $S\times S$ matrix whose
entries lie in $\{1,2,3,\ldots,\infty\}$ and satisfy $m_{ss}=1$ 
for all $s\in S$, $m_{st}\geqslant 2$ if $s\neq t$.
A Coxeter matrix determines a \emph{Coxeter group}
\[
	W=\Big\langle
		S
		\ \Big|\ 
		s^2=e\text{ for }s\in S,
		\ \underbrace{sts\cdots}_{m_{st}\text{ terms}} 
		= \underbrace{tst\cdots}_{m_{st}\text{ terms}}
		\text{ for }s,t\in S
	\Big\rangle
\]
When $m_{st}=\infty$ no relation is applied.
The relations
\[
	\underbrace{sts\cdots}_{m_{st}\text{ terms}} 
		= \underbrace{tst\cdots}_{m_{st}\text{ terms}}
\]
are called the \emph{braid relations} or \emph{braid moves}.
The pair $(W,S)$ is called a \emph{Coxeter system}.

\begin{example}[The Coxeter system of type $A_{n-1}$]\label{example-coxeter-An}
	Let $W=\S_n$ be the symmetric group on $n$ letters,
	and let $S_n=\{s_1,\ldots,s_{n-1}\}$ where $s_i$
	is the adjacent transposition $(i\ i+1)$.
	Then $(\S_n,S_n)$ is a Coxeter system,
	called the \emph{Coxeter system of type $A_{n-1}$}.
	Observe that:
	\[
		m_{s_is_j}=\begin{cases}
			2 & \text{ if }|i-j|>1
			\\
			3 & \text{ if }|i-j|=1
		\end{cases}
	\]
	Indeed, if $|i-j|>1$ then $s_i$ and $s_j$ are disjoint
	transpositions, so that $s_is_j$ has order $2$,
	while if $|i-j|=1$ then $s_is_j$ is a $3$-cycle, and so
	has order $3$.
	(The observation really just shows that if $W$ is the Coxeter
	group of this type, then the relevant relations hold in $\S_n$
	so that there is a surjection $W\to\S_n$.  To show that this is
	an isomorphism, one must show that the relations of the Coxeter group
	are sufficient to relate any two words representing the same element
	of $\S_n$.  That is a simple exercise.)
\end{example}

Let $(W,S)$ be a Coxeter system.
A \emph{word} in $S$ is a tuple $\bfs=(s_1,\ldots,s_l)$ of elements of $S$.
We say that $w=w(\bfs)=s_1\ldots s_l$ 
is the element \emph{represented} by $\bfs$,
and we say equivalently that $\bfs$ is an \emph{expression} for $w=w(\bfs)$.
The \emph{length} of an element $w\in W$, denoted $\ell(w)$, 
is the minimum length of a word representing $w$.
We say that $\bfs$ is a \emph{reduced expression} for $w$ if it is a word
of minimum length representing $w$.

We will often blur the difference between words and their expressions,
writing $w=s_1\cdots s_l$ for an element of $W$, and referring to
$s_1\cdots s_l$ as a \emph{word} or  \emph{expression} for $w$, 
hoping that it will be clear from what is written 
that the expression $(s_1,\ldots,s_l)$ is to be understood.

Given an element $w\in W$, there are two possibilities for
$\ell(sw)$:
\begin{itemize}
	\item
	$\ell(sw) = \ell(w)+1$.
	In this case one can obtain a reduced expression for $sw$
	by putting $s$ in front of a reduced expression for $w$.

	\item
	$\ell(sw) = \ell(w)-1$.
	In this case $w$ has a reduced expression beginning
	with $s$.
\end{itemize}
(See~\cite[pp.35-36]{Davis}.)

Here are two important results on reduced words in Coxeter groups.

\begin{theorem}[{Matsumoto's theorem~\cite{Matsumoto}, 
\cite[section 1.2]{GeckPfeiffer}}] \label{theorem-matsumoto}
	Let $(W,S)$ be a Coxeter system.  Then
	any reduced expression for an element of $W$
	can be transformed into any other by repeatedly
	replacing subwords of the form
	$sts\cdots$ (with $m_{st}$ terms)
	with $tst\cdots$ (again with $m_{st}$ terms).
\end{theorem}

\begin{theorem}[The word problem, {Tits~\cite{Tits}, \cite[3.4.2]{Davis}}]
\label{theorem-word}
	Let $(W,S)$ be a Coxeter system.  Then
	a word in $S$ is a reduced expression 
	if and only if it cannot be shortened
	by applying a sequence of the following \emph{M-operations}
	or \emph{M-moves}:
	\begin{itemize}
		\item
		Delete a subword of the form $ss$.
		\item
		Replace a subword of the form 
		$\underbrace{sts\cdots}_{m_{st}\text{ terms}}$
		with
		$\underbrace{tst\cdots}_{m_{st}\text{ terms}}$.
	\end{itemize}
	Any two reduced expressions for the same element
	differ only by a sequence of moves of the second kind.
\end{theorem}

Let $(W,S)$ be a Coxeter system.
Let $T\subseteq S$.
The associated \emph{special subgroup} is the subgroup of $W$ 
generated by $T$, and is denoted $W_T$.
The pair $(W_T,T)$ is then a Coxeter system, which is to say,
$W_T$ is precisely the Coxeter group with generators $T$ and
with Coxeter matrix obtained from the Coxeter matrix of $(W,S)$
in the evident way~\cite[4.1.6]{Davis}.

\subsection{Cosets in Coxeter groups}\label{subsection-cosets}
For the material in this subsection we refer to section~2.1
of~\cite{GeckPfeiffer} and section~4.3 of~\cite{Davis}.

Let $(W,S)$ be a Coxeter system, and let $J\subseteq S$.
The cosets $W_J\backslash W$ are the subject of the following theory,
which will be extremely useful to us.
Define
\[
	X_J = \{w\in W\mid \ell(sw)>\ell(w)\text{ for all }s\in J\}.
\]
Thus $X_J$ consists of all elements of $W$ that have no 
reduced expressions beginning with an element of $J$.
The elements of $X_J$ are called \emph{$(J,\emptyset)$-reduced},
and referred to as the \emph{distinguished right coset representatives}
for $W_J$, for reasons that the next theorem
will make clear.
If $J\subseteq K\subseteq S$, then we write $X^K_J$ for the set of distinguished
right-coset representatives for $W_J$ in $W_K$.

\begin{theorem}\label{theorem-distinguished}
	\begin{enumerate}
		\item
		$x\in X_J$ if and only if $\ell(vx) = \ell(v)+\ell(x)$
		for all $v\in W_J$.
		\item
		For each $w\in W$ there exist unique $x\in X_J$
		and $v\in W_J$ such that $w = vx$.
		\item
		$X_J$ forms a complete set of representatives for
		$W_J\backslash W$.
		\item
		If $x\in X_J$ then $x$ is the unique shortest
		element in $W_J x$.
	\end{enumerate}
\end{theorem}

There is a similar theory for the cosets $W/W_J$, in which the role of
$X_J$ is now played by $X_J^{-1}$, elements of which are called
\emph{$(\emptyset,J)$-reduced}.

Moreover, if $J,K\subseteq S$ then
there is also a theory for the double cosets $W_J\backslash W/W_K$.
We define $X_{JK} = X_J\cap X_K^{-1}$.  
Thus an element $x\in W$ lies in $X_{JK}$ if and only
if it has no reduced expressions beginning with a letter in $J$
or ending with a letter in $K$.
The elements of $X_{JK}$
are called \emph{distinguished double coset representatives of
$W_J$ and $W_K$ in $W$}, and we also refer to them as
\emph{$(J,K)$-reduced}.  They form a complete set of 
representatives for the double cosets $W_J\backslash W/W_K$,
and each one is the unique shortest element in its double coset.

The \emph{Mackey decomposition} states that for $J,K\subseteq S$,
\[
	X_J 
	= \bigsqcup_{d\in X_{JK}} d\cdot X^K_{J^d\cap K}	
\]
Inverting the Mackey decomposition gives us a version for
the left cosets
\[
	X_J^{-1} 
	= \bigsqcup_{d\in X_{KJ}} (X^K_{K\cap \prescript{d}{}\!J})^{-1}\cdot d
\]
Here, as is common in group theory,
the notation $J^d$ and $\prescript{d}{}\!J$ denotes conjugation
by $d$, with the positive power of $d$ on the side indicated by
the notation, so that
\[
	J^d=\{d^{-1}jd\mid j\in J\},
	\qquad
	\prescript{d}{}\!J = \{djd^{-1}\mid j\in J\}.
\]
Observe that in both Mackey decompositions the lengths add in products.
For example, suppose that $d\in X_{JK}$ and $y\in X^K_{J^d\cap K}$,
so that $d\in X_K^{-1}$ and $y\in W_K$, and hence $\ell(dy) = \ell(d)+\ell(y)$.

\subsection{Iwahori-Hecke algebras}\label{subsection-iwahori-hecke}

Let $(W,S)$ be a Coxeter system, let $R$ be a commutative ring,
and let $q\in R^\times$ be a unit.  
The \emph{Iwahori-Hecke algebra} associated to $(W,S)$ is the algebra 
$\H_W$ with generators 
\[
	T_s\text{ for }s\in S
\]
and relations:
\begin{align*}
	\underbrace{T_sT_tT_s\cdots}_{m_{st}\text{ terms}} 
	&= \underbrace{T_tT_sT_t\cdots}_{m_{st}\text{ terms}}
	&& 
	\text{for }s,t\in S
	\\
	(T_s+1)(&T_s-q)= 0 
	&& 
	\text{for }s\in S
\end{align*}

\begin{example}[The Iwahori-Hecke algebra of type $A_{n-1}$]
\label{example-ih-An}
	Take $W=\S_n$ and $S_n=\{s_1,\ldots,s_{n-1}\}$,
	as in Example~\ref{example-coxeter-An},
	so that $(W,S)=(\S_n,S_n)$ is the Coxeter system
	of type $A_{n-1}$.
	Then $m_{s_is_j}=2$ if $|s_i-s_j|>1$, and 
	and $m_{s_is_j}=3$ if $|s_i-s_j|=1$, so that $\H_{\S_n}$
	has generators
	\[
		T_{s_1},\ldots,T_{s_{n-1}}
	\]
	and relations
	\begin{align*}
		T_{s_i}T_{s_j} &= T_{s_j}T_{s_i}  
		&& \text{for }|i-j|>1,
		\\
		T_{s_i}T_{s_j}T_{s_i} &= T_{s_j}T_{s_i}T_{s_j}  
		&& \text{for }|i-j|=1,
		\\
		(T_{s_i}+1)(&T_{s_i}-q)= 0 && \text{for all }i.
	\end{align*}
	Thus, if we write $T_i=T_{s_i}$, then $\H_{\S_n}$
	becomes exactly the algebra $\H_n$ defined in the introduction.
\end{example}

Let $w\in W$, and let $w=s_1\cdots s_r$ be any reduced expression for 
$w$.  Then by Matsumoto's Theorem~\ref{theorem-matsumoto}, the quantity 
\[
	T_w = T_{s_1}T_{s_2}\cdots T_{s_r}
\]
depends only on $w$ and not on the reduced expression.
Suppose that $u,v\in W$ satisfy $\ell(uv)=\ell(u)+\ell(v)$.
Then one can obtain a reduced expression for $uv$ by combining
reduced expressions for $u$ and $v$.  We therefore obtain:
\[
	T_uT_v = T_{uv}
	\text{ for }u,v\in W\text{ such that }
	\ell(uv)=\ell(u)+\ell(v)
\]
The significance of the elements $T_w$ is the following central
result, for which see chapter IV, section~2, exercise~23 of~\cite{Bourbaki},
or Theorem~4.4.6 of~\cite{GeckPfeiffer}, 
or Theorem~1.13 of~\cite{Mathas} for the case $\H_W=\H_n$.

\begin{theorem}[Basis theorem]\label{theorem-basis}
	The elements $T_w$ for $w\in W$ form a basis for $\H_W$
	as an $R$-module, called the \emph{standard basis}.
\end{theorem}

And we have the following consequence, which is extremely
important for the present paper.

\begin{proposition}\label{proposition-basis}
	Let $(W,S)$ be a Coxeter system and let
	$J\subseteq S$.  Then $\H_W$ is free as a left $\H_{W_J}$-module
	with basis $\{T_x\mid x\in X_J\}$.
	In particular, $\t\otimes_{\H_{W_J}}\H_W$
	is free with basis $\{1\otimes T_x\mid x\in X_J\}$.

	Similarly, $\H_W$ is free as a right $\H_{W_J}$-module
	with basis $\{T_x\mid x\in X_J^{-1}\}$,
	and $\H_W\otimes_{\H_{W_J}}\t$ is free with 
	basis $\{T_x\otimes 1\mid x\in X_J^{-1}\}$.
\end{proposition}

This follows by combining Theorems~\ref{theorem-distinguished}
and~\ref{theorem-basis}.  The point is that there is a bijection
$W_J\times X_J\to W$, $(v,x)\mapsto vx$ satisfying $\ell(vx)=\ell(v)+\ell(x)$
for every $(v,x)\in W_J\times X_J$, so that $T_{vx}=T_vT_x$.
See~\cite[4.4.7]{GeckPfeiffer}.

\section{Symmetric groups and the complex of injective words}
\label{section-injective-words}

In this section, we will recall the definition of the complex
of injective words, and we will give a proof that it is
highly acyclic.  This result is originally due to 
Farmer~\cite{Farmer}, and has since been proved 
in different ways by many authors,
including Bj\"orner-Wachs~\cite{BjornerWachs},
Kerz~\cite{Kerz}, and Randal-Williams~\cite{RandalWilliamsConfig}.
The approach that we present here is closest to that of Kerz,
but tailored to our later extension to Iwahori-Hecke algebras.
Throughout the section we fix a commutative ring $R$.

If $A$ is a set, then an \emph{injective word} on $A$ is an ordered
tuple $(a_0,\ldots,a_r)$ of elements of $A$ such that no element
appears more than once.  We allow the empty word $()$.

\begin{definition}[The complex of injective words]
\label{definition-cn}
	Let $n\geqslant 0$.
	The \emph{complex of injective words} $\C(n)$
	is the chain complex,
	concentrated in degrees $ -1\leqslant r\leqslant n-1$, 
	that in degree $r$ is the $R$-module with basis 
	consisting of the injective words 
	$(a_0,\ldots,a_r)$ of length $(r+1)$ on the set $\{1,\ldots,n\}$.
	The differential $\partial^r\colon\C(n)_r\to\C(n)_{r-1}$
	is defined to be given by the alternating sum
	\[
		\partial^r(a_0,\ldots,a_r)
		=
		\sum_{j=0}^r(-1)^j(a_0,\ldots,\widehat{a_j},\ldots,a_r).
	\]
	We regard $\C(n)$ as a chain complex of $\S_n$-modules
	by allowing $\S_n$ to act on the letters of a word in the evident way.
	Note that $\C(n)_{-1}$ is a copy of $R$ generated by the empty
	word $()$.
\end{definition}

\begin{remark}
	The complex of injective words appears in many forms,
	for example as the realisation of a poset in~\cite{Farmer}
	and~\cite{BjornerWachs}, a chain complex in~\cite{Kerz},
	and as a semisimplicial set in~\cite{RandalWilliamsConfig}.  
	We are working in the linear setting of $R\S_n$-modules,
	and so our complex is a chain complex of $R\S_n$-modules.
\end{remark}

\begin{note}
	Throughout the paper we will use notation like
	$\partial^r$ in Definition~\ref{definition-cn},
	where the superscript indicates the degree in which
	the differential originates.
	This causes visual clutter and is sometimes
	extraneous, but will be extremely helpful 
	later on in keeping track of degrees.
\end{note}

\begin{theorem}[{Farmer~\cite{Farmer}}]\label{theorem-farmer}
	$H_d(\C(n))=0$ for $d\leqslant n-2$.
\end{theorem}

In order to prove this theorem we will 
define a filtration of $\C(n)$ by looking at the position of
the letter $n$.  This is essentially the technique used by Kerz~\cite{Kerz}.

\begin{definition}[The filtration of $\C(n)$]
\label{definition-symmetric-filtration}
	Let $0\leqslant p\leqslant n-1$.
	Define $F_p\subseteq \C(n)$ to be the subcomplex of $\C(n)$
	spanned by all words for which the letter
	$n$ appears in the last $(p+1)$ places, or not at all.
	Thus we obtain a filtration
	\[
		F_0\subseteq F_1\subseteq\cdots \subseteq F_{n-1}=\C(n).
	\]
	Observe that the $F_p$ are not submodules with respect
	to the $\S_n$ action, since that can change the position of $n$,
	but that they are submodules with respect to the restricted
	action of $\S_{n-1}$.
\end{definition}

The following notation fixes our conventions for cones and suspensions
of chain complexes.  The conventions are chosen so as to make the
subsequent parts of the proof as direct as possible.

\begin{definition}\label{definition-cones-suspensions}
	Let $X$ be a chain complex with differentials $d^r_X$.
	The \emph{cone} on $X$, denoted $CX$, is the chain complex defined
	by
	\[
		(CX)_r=X_r\oplus X_{r-1}	
	\]
	with
	\[
		d_{CX}^r\colon (CX)_r\longrightarrow (CX)_{r-1}
	\]
	defined by $d_{CX}^r(x,y)= (d_X^r(x)+(-1)^r y,d_X^{r-1}(y))$.
	The \emph{suspension} $\Sigma X$ is the chain complex defined
	by
	\[
		(\Sigma X)_r = X_{r-1}
	\]
	with
	\[
		d_{\Sigma X}^r\colon(\Sigma X)_r
		\longrightarrow
		(\Sigma X)_{r-1}
	\]
	defined by $d_{\Sigma X}^r = d_X^{r-1}$.
\end{definition}

\begin{lemma}\label{lemma-fzero}
	There is an isomorphism
	\[
		C(\C(n-1))\xrightarrow{\ \cong\ }F_0
	\]
	of chain complexes of $R\S_n$-modules.
\end{lemma}

\begin{proof}
	$F_0$ is the span of all words in which either $n$ does not appear,
	or appears in the final position.
	We define a map
	\[
		C(\C(n-1))\longrightarrow F_0
	\]
	in degree $r$ by
	\[
		((x_0,\ldots,x_r),0)
		\longmapsto
		(x_0,\ldots,x_r),
		\qquad
		(0,(y_0,\ldots,y_{r-1}))
		\longmapsto
		(y_0,\ldots,y_{r-1},n),
	\]
	for $(x_0,\ldots,x_r)\in\C(n-1)_r$
	and $(y_0,\ldots,y_{r-1})\in\C(n-1)_{r-1}$.
	Since $F_0$ consists of words in which $n$ appears in the final
	position or not at all, this is an isomorphism,
	and it is straightforward to check that it 
	commutes with the differentials.
\end{proof}

\begin{lemma}\label{lemma-fp}
	Let $1\leqslant p\leqslant n-1$.
	There is an isomorphism 
	\[
		R\S_{n-1}\otimes_{R\S_{n-p-1}}\Sigma^{p+1}\C(n-p-1)
		\xrightarrow{\ \cong\ }
		F_p/F_{p-1}
	\]
	of chain complexes of $R\S_{n-1}$-modules.
	In particular, as a chain complex of $R$-modules,
	$F_p/F_{p-1}$ is isomorphic to 
	a direct sum of finitely many copies of $\Sigma^{p+1}\C(n-p-1)$.
\end{lemma}

\begin{proof}
	Let $-1\leqslant r\leqslant n-p-2$.
	Then in degree $(p+1)+r$, 
	$F_p$ is the span of all words with $n$ in the last $(p+1)$
	places or not at all, while $F_{p-1}$ is the span of all
	words with $n$ in the last $p$ places or not at all,
	so that $F_p/F_{p-1}$ has a basis consisting
	of all injective words of the form $(y_0,\ldots,y_r,n,z_1,\ldots,z_p)$,
	i.e.~with $n$ in precisely the $(p+1)$-st place from
	the end.
	The effect of the boundary map $\partial^{(p+1)+r}$ 
	on such a word is	
	\[
		(y_0,\ldots,y_r,n,z_1,\ldots,z_p)
		\longmapsto
		\sum_{j=0}^r(-1)^j(y_0,\ldots,\widehat{y_j},\ldots,y_r,
		n,z_1,\ldots,z_p).
	\]
	Here the last $(p+1)$ summands of $\partial^{(p+1)+r}$,
	in which one of the last $(p+1)$ letters is deleted, are not
	present because any resulting word would lie in $F_{p-1}$.
	So the differential affects the segment $(y_0,\ldots,y_r)$ 
	in the usual way, while leaving the remaining segment  unchanged.
	
	We now define our map in degree $(p+1)+r$, 
	for $-1\leqslant r\leqslant n-p-2$, by 
	\[
		\sigma\otimes (x_0,\ldots,x_r)
		\longmapsto
		\sigma(x_0,\ldots,x_r,n,n-p,\ldots,n-1)
	\]
	for $\sigma\in\S_{n-1}$ and 
	$(x_0,\ldots,x_r)\in\Sigma^{p+1}\C(n-p-1)_{(p+1)+r}
	=\C(n-p-1)_r$,
	where the final $p$ terms increase from $n-p$ to $n-1$.
	This is well-defined, and
	the fact that it is a chain map follows quickly from the 
	computation above.
	To see that it is an isomorphism we choose, for each injective word
	$z=(z_1,\ldots,z_p)$ on $\{1,\ldots,n-1\}$, an element
	$\sigma_z\in\S_{n-1}$ for which $z=\sigma_z(n-p,\ldots,n-1)$.
	Then the $\sigma_z$ are a set of representatives for the cosets
	$\S_{n-1}/\S_{n-p-1}$, so that every $\sigma\otimes (x_0,\ldots,x_r)$
	in $R\S_{n-1}\otimes_{R\S_{n-p-1}}\Sigma^{p+1}\C(n-p-1)_r$
	can be rewritten in the form 
	$\sigma_z\otimes(x'_0,\ldots,x'_r)$
	for a unique choice of $z$ and $(x'_0,\ldots,x'_r)$.
	An inverse to our map can then be given by
	the rule
	\[
		(x_0,\ldots,x_r,n,z_1,\ldots,z_p)
		\longmapsto \sigma_z\otimes\sigma_z^{-1}(x_0,\ldots,x_r).
		\qedhere
	\]
\end{proof}

\begin{proof}[Proof of Theorem~\ref{theorem-farmer}]
	This is proved by induction on $n\geqslant 0$.
	In the case $n=0$,  $\C(0)$ consists only of a copy of $R$ in
	degree $-1=n-1$, and its homology has the same description.
	Now take $n>0$ and suppose that the claim holds for all smaller values
	of $n$.  Then $\C(n)$ has filtration 
	$F_0\subseteq\cdots\subseteq F_{n-1}$.
	The subcomplex $F_0\cong C(\C(n-1))$ is chain-contractible
	by Lemma~\ref{lemma-fzero}.
	And the induction hypothesis tells us that each
	$\C(n-p-1)$ has zero homology in all degrees up to and including
	$(n-p-1)-2$, so that 
	$R\S_{n-1}\otimes_{R\S_{n-p-1}}\Sigma^{p+1}\C(n-p-1))$
	has zero homology in degrees up to and including $(n-p-1)-2+(p+1)=n-2$,
	so that by Lemma~\ref{lemma-fp} the same is true of $F_p/F_{p-1}$.
	Since $F_0$ and all $F_p/F_{p-1}$ have vanishing homology
	in the stated range,
	the same follows for $\C(n)$ itself.
\end{proof}

\section{$\C(n)$ via the group ring and Coxeter generators}
\label{section-translation}

In the last section we recalled the complex of injective words,
and we gave a proof that its homology vanishes up to and including degree
$n-2$.
These are key ingredients in
the proof of homological stability for the symmetric groups, and 
in order to prove stability for the Iwahori-Hecke algebras
we must extend them to that setting.
However, the Iwahori-Hecke algebra $\H_n$
is obtained by taking the presentation of
the group ring $R\S_n$, 
where $\S_n$ is regarded as the group ring of a Coxeter group,
and deforming one of the relations, so that we must 
begin by recasting the complex of injective words in the same terms.
In this section we will rephrase the chain complex
$\C(n)$, the filtration
$F_0\subseteq\cdots\subseteq F_{n-1}=\C(n)$, and the isomorphisms
$F_0\cong C(\C(n-1))$ and 
$F_p/F_{p-1}\cong R\S_{n-1}\otimes_{R\S_{n-p-1}}\Sigma^{p+1}\C(n-p-1)$
in terms of the group ring of $\S_n$, regarded as a Coxeter group.
The result will be a new chain complex
$\C'(n)$, with a filtration
$F_0'\subseteq\cdots\subseteq F_{n-1}'=\C(n)$, and isomorphisms
$F_0'\cong C(\C'(n-1))$ and 
$F_p'/F_{p-1}'\cong R\S_{n-1}\otimes_{R\S_{n-p-1}}\Sigma^{p+1}\C'(n-p-1)$,
all identified with the originals under an appropriate isomorphism,
and all described only in terms of group rings and Coxeter generators.
This process of `algebra-ising' the existing construction
is a key step in the work of this paper, 
serving as motivation and explanation for everything that follows,
and allowing us to generalise directly to Iwahori-Hecke algebras.

The account that follows will make repeated use of a particular family of
elements of $R\S_n$, which we define here.

\begin{definition}[The elements $s_{ba}$]\label{definition-sba}
	Given $n\geqslant b\geqslant a\geqslant 1$, we define
	\[s_{ba} = s_{b-1}s_{b-2}\cdots s_a.\]
	Note that the indices always decrease from left to right,
	so that if $b=a$, then the product is empty and $s_{ba}$ is the
	unit element.
	When necessary for clarity, we will separate the 
	indices with a comma, i.e.~$s_{b,a}=s_{ba}$.
	Note that $s_{a+1,a}=s_a$. 
\end{definition}

\begin{remark}\label{remark-sab-permutation}
	The element $s_{ba}$ can be described in cycle notation as
	\begin{eqnarray*}
		s_{ba}
		&=&s_{b-1}\cdots s_{a}\\
		&=& (b\, b-1)(b-1\, b-2)\cdots(a+1\,a)\\
		&=& (b\,b-1\,\cdots\,a),
	\end{eqnarray*}
	and in two-line notation as
	\[
		s_{ba}
		=
		\begin{pmatrix}
			1 & \cdots & a-1
			&
			a & a+1 & \cdots& b
			&
			b+1 & \cdots & n
			\\
			1 & \cdots & a-1
			&
			b & a & \cdots & b-1
			&
			b+1 & \cdots & n
		\end{pmatrix}.
	\]
	Thus $s_{ba}$ is the permutation that decreases
	each of $a+1,\ldots,b$ by $1$, and that sends $a$ to $b$.
	The elements $s_{ba}$ will be used to 
	encode deletion of letters, 
	and to move the letter $n$ into particular positions required by
	the filtration.
	See Remarks~\ref{remark-deleting-letters} 
	and~\ref{remark-moving-letters} below.
\end{remark}

Before continuing, it will be useful to
explain how the complex of injective words 
$\C(n)$ fits into the theory of augmented semi-simplicial objects.
An \emph{augmented semi-simplicial object} $A_\bullet$ in some category
is defined to be a sequence of objects $A_{-1},A_0,A_1,\ldots $
in that category, together with \emph{face maps}
$\partial^r_j\colon A_r\to A_{r-1}$ for $r\geqslant 0$
and $j=0,\ldots,r$, which satisfy the relations
$\partial^{r-1}_i\partial^r_j
=
\partial^{r-1}_{j-1}\partial^r_i$
for $r\geqslant 1$ and $0\leqslant i<j\leqslant r$.
If $A_\bullet$ is an augmented semi-simplicial object $A_\bullet$
in an \emph{abelian} category, then we can associate to it a
chain complex $A_\ast$ given in degree $r$ by $A_r$,
and with differential $\partial^r=\sum_{j=0}^r(-1)^j\partial^r_j$.

Recall that $\C(n)$ 
in degree $r$ is the free $R$-module with basis the
injective words on the set $\{1,\ldots,n\}$,
and that the differential
$\partial^r\colon\C(n)_r\to\C(n)_{r-1}$ sends 
$(a_0,\ldots,a_r)$ to the alternating sum 
$\sum_{j=0}^r(-1)^j(a_0,\ldots,\widehat{a_j},\ldots,a_r)$.
Thus we may write $\partial^r=\sum_{j=0}^r(-1)^j\partial^r_j$,
where $\partial^r_j$ is defined by
$\partial^r_j(a_0,\ldots,a_r)=(a_0,\ldots,\widehat{a_j},\ldots,a_r)$.
This makes $\C(n)$ into the chain complex associated to the
{augmented semi-simplicial $R\S_n$-module} $\C(n)_\bullet$
with {face maps} $\partial^r_j$.

Augmented semi-simplicial objects give us an organised way
of splitting differentials into their component summands,
and this will be useful throughout the rest of the paper.
We will point out such semi-simplicial structures where they
occur.

We are now in a position to define the complex $\C'(n)$,
which is our rephrasing of $\C(n)$ in terms of the group algebra of 
$\S_n$, regarded as a Coxeter group.

\begin{definition}[The chain complex $\C'(n)$]\label{definition-cdashn}
	Let $\C'(n)$ denote the chain complex
	of $R\S_n$-modules, concentrated in degrees $r=-1,\ldots,n-1$,
	which is given in degree $r$ by
	\[	
	\C'(n)_r=R\S_n\otimes_{R\S_{n-r-1}}\t,
	\]
	and whose differential $\partial^r\colon\C'(n)_r\to\C'(n)_{r-1}$
	is defined by
	\[
		\partial^r = \sum_{j=0}^r (-1)^j\partial^r_j,
		\qquad
		\partial^r_j(\sigma\otimes 1)
		=
		\sigma s_{n-r+j,n-r}\otimes 1.
	\]
	Thus $\C'(n)$ is the chain complex associated to an augmented
	semi-simplicial $R\S_n$-module with face maps $\partial^r_j$.
\end{definition}

\begin{remark}\label{remark-similarity}
	Definition~\ref{definition-cdashn} 
	is similar to several that already appear in the literature,
	for example
	Definition~34 of~\cite{HepworthCoxeter},
	Definition~7.5 of~\cite{Krannich} and Definition~7.1 of~\cite{Boyd}.
\end{remark}

\begin{definition}[The isomorphism $\Theta$]
	Define $\Theta\colon\C'(n)\to\C(n)$ 
	in degree $r$ by
	\[
		\Theta_r(\sigma\otimes 1)
		=\sigma(n-r,\ldots,n)
		=(\sigma(n-r),\ldots,\sigma(n)).
	\]
\end{definition}

\begin{proposition}\label{proposition-Theta}
	$\Theta$ is an isomorphism of chain complexes of $R\S_n$-modules.
\end{proposition}

\begin{proof}
	Recall that $\C(n)_r$ is the $R$-linear span of all injective words 
	$(a_0,\ldots,a_r)$ on the set $\{1,\ldots,n\}$.
	The group $\S_n$ acts transitively on the set of injective words,
	and the stabiliser of the word $(n-r,n-r+1,\ldots,n)$ 
	is exactly $\S_{n-r-1}\subseteq\S_n$.
	It follows that $\Theta_r$ is a well-defined isomorphism
	in each degree.

	Next we must check that $\Theta_r$ commutes with the differentials.
	It is enough to check that it commutes with the $\partial^r_j$,
	i.e.~that $\partial^r_j\Theta_r = \Theta_{r-1}\partial^r_j$,
	and so we compute
	\begin{align*}
		\partial^r_j
		\Theta_r(\sigma\otimes 1)
		&=
		\partial^r_j
		(\sigma(n-r),\ldots,\sigma(n))
		\\
		&=
		(\sigma(n-r),\ldots,\widehat{\sigma(n-r+j)},\ldots,\sigma(n))
		\\
		&=
		(\sigma(s_{n-r+j,n-r}(n-r+1)),\ldots,\sigma(s_{n-r+j,n-r}(n)))
		\\
		&=
		\Theta_{r-1}(\sigma s_{n-r+j,n-r}\otimes 1)
	\end{align*}
	as required.
	Here recall from Remark~\ref{remark-sab-permutation}	
	that $s_{n-r+j,n-r}$ is the cycle $(n-r+j,\cdots, n-r)$,
	which decreases each of $n-r+1,\ldots,n-r+j$ by $1$
	while preserving each of $n-r+j+1,\ldots,n$.
	(Note that the proof in fact demonstrates that $\Theta$ 
	is an isomorphism
	of augmented semi-simplicial $R\S_n$-modules.)
\end{proof}

\begin{remark}[The element $s_{n-r+j,n-r}$ and deleting letters]
\label{remark-deleting-letters}
	In $\C(n)$ the map $\partial^r_j$ is given by erasing the $j$-th letter,
	$\partial^r_j(a_0,\ldots,a_r)=(a_0,\ldots,\widehat{a_j},\ldots,a_r)$,
	whereas in $\C'(n)$ the map $\partial^r_j$ is given by
	$\partial^r_j(\sigma\otimes 1) = \sigma s_{n-r+j,n-r}\otimes 1$.
	How do these two correspond, and in particular,
	why does $s_{n-r+j,n-r}$ appear?
	Well, in the last expression we had
	$\sigma\otimes 1\in R\S_n\otimes_{R\S_{n-r-1}}\t$
	and
	$\sigma s_{n-r+j,n-r}\otimes 1\in R\S_n\otimes_{R\S_{n-r}}\t$.
	Increasing the algebra over which the 
	tensor product is taken, from $R\S_{n-r-1}$ to $R\S_{n-r}$,
	corresponds under $\Theta$ 
	to erasing the \emph{first} letter of a word.
	The role played by $s_{n-r+j,n-r}$, then, is to take the $j$-th letter
	$\sigma(n-r+j)$, which is to be deleted, and move it down to 
	start of the word, where it is indeed 
	deleted thanks to the change in the tensor product.
\end{remark}

Having defined $\C'(n)$ and the isomorphism
$\Theta\colon\C'(n)\xrightarrow{\cong}\C(n)$, 
we will now transport the filtration of $\C(n)$ to $\C'(n)$.

\begin{definition}\label{definition-Fpdash}
	Let $F'_0\subseteq\cdots\subseteq F'_{n-1}=\C'(n)$
	be the filtration defined by letting
	$F'_p$ in degree $r$ be the $R\S_{n-1}$-span
	of all elements of the following two kinds:
	\begin{enumerate}
		\item
		In the cases $r=-1,\ldots, n-2$,
		the element 
		\[
			s_{n,n-r-1}\otimes 1.
		\]
		\item
		In the cases $r= 0,\ldots, n-1$, 
		the elements 
		\[
			s_{n,n-t}\otimes 1
		\]
		for $t=0,\ldots,\min(r,p)$,
	\end{enumerate}
\end{definition}

It may seem unnecessary to distinguish between the two kinds of element
above, especially when $p\geqslant r$,
given the similarity between $s_{n,n-r-1}\otimes 1$
and the $s_{n,n-t}\otimes 1$.
But in fact it will be important to maintain the distinction,
see Remark~\ref{remark-distinction}.

\begin{proposition}\label{proposition-Fpdash}
	Under the isomorphism $\Theta$, the filtration
	$F'_0\subseteq\cdots\subseteq F'_{n-1}=\C'(n)$
	corresponds precisely to the filtration
	$F_0\subseteq\cdots\subseteq F_{n-1}=\C(n)$.
\end{proposition}

\begin{proof}
	Recall that $F_p\subseteq\C(n)$ is given in degree $r$ by the span
	of all injective words $(x_0,\ldots,x_r)$ 
	in which the letter $n$
	appears in the final $p+1$ positions, or not at all.
	This means that $F_p$ consists not of $\S_n$-modules but of
	$\S_{n-1}$-modules.
	The injective words in $F_p$ consist of various 
	$\S_{n-1}$-orbits, determined by the position of the letter $n$.
	These orbits come to us in two distinct kinds,
	the first containing just a single orbit, and the second
	containing a family of orbits:
	\begin{enumerate}
		\item
		The single orbit of words in which 
		$n$ does not appear at all.
		This orbit is only present in cases 
		$-1\leqslant r\leqslant n-2$,
		for if $r=n-1$ then our word has length $n$ and all letters
		must appear.
		\item
		The orbits in which $n$ appears in position $r-t$, 
		for $t=0,\ldots, \min(r,p)$.
		Such orbits are only present when $0\leqslant r\leqslant n-1$,
		for when $r=-1$ the only word is the empty word,
		and $n$ does not appear there.
	\end{enumerate}
	We now choose representatives of these orbits as follows:
	\begin{enumerate}
		\item
		In the cases $r=-1,\ldots, n-2$,
		we represent the $\S_{n-1}$-orbit in which $n$ does not appear
		by the word 
		\[(n-r-1,\ldots,n-1).\]
		\item
		In the cases $r= 0,\ldots, n-1$, for $t=0,\ldots,\min(r,p)$
		we represent the orbit in which $n$ appears 
		in position $r-t$ by the word
		\[(n-r,\ldots,n-t-1,n,n-t,\ldots,n-1).\]
		This is the word in which $n-r,\ldots,n-1$ appear in order,
		with $n$ inserted in position $r-t$.
	\end{enumerate}
	So $F_p$ is the $R\S_{n-1}$-span of the words just listed.
	
	To show that $\Theta F'_p =F_p$, it will suffice
	to check that the generators of $F'_p$ listed in 
	Definition~\ref{definition-Fpdash} are mapped to the generators 
	listed above.
	And indeed, 
	$s_{n,n-r-1}$ is the cycle $(n\ n-1\ \cdots\ n-r-1)$, so that
	\[
		\Theta_r(s_{n,n-r-1}\otimes 1)
		=
		s_{n,n-r-1}(n-r,\ldots,n)
		=
		(n-r-1,\ldots,n-1)
	\]
	is our representative of the first kind of orbit.
	And $s_{n,n-t}$ is the cycle $(n\ n-1\ \cdots\ n-t)$, so that
	\[
		\Theta_r(s_{n,n-t}\otimes 1)
		=
		s_{n,n-t}(n-r,\ldots,n)
		=
		(n-r,\cdots,n-t-1,n,n-t,\cdots,n-1)
	\]
	is our representative of the second variety of orbit.
\end{proof}

\begin{remark}[The role of $s_{n,n-r-1}$ and $s_{n,n-t}$]
\label{remark-moving-letters}
	In this proof $s_{n,n-r-1}$ was used to
	produce a word in which $n$ does not appear, and it did this
	by moving $n$ into a position where it would be deleted under the
	relevant tensor product, 
	much as in Remark~\ref{remark-deleting-letters}.
	Somewhat differently, $s_{n,n-t}$ was used
	to move the letter $n$ into position $r-t$ in our word
	while leaving the order of the remaining letters unaffected.
\end{remark}

\begin{remark}[A distinction]\label{remark-distinction}
	In the last definition and proposition we see two instances of the same
	important distinction,
	namely the distinction between the two kinds of generator 
	of $F'_p$ in Definition~\ref{definition-Fpdash},
	and the distinction between the two kinds of orbit
	in the proof of Proposition~\ref{proposition-Fpdash}.
	Despite the similarity between the elements
	$s_{n,n-r-1}$ and $s_{n,n-t}$,
	it will be necessary to keep careful track of the distinction 
	between the two.
	For example, under the identification $C(\C(n-1))\cong F_0$,
	the distinction between the two kinds 
	becomes the distinction between the two parts of the cone, 
	namely the unshifted `base' part and the shifted `cone' part.
\end{remark}

We will now transport the identifications of the filtration quotients
of $\C(n)$ to the setting of $\C'(n)$.
Recall from Definition~\ref{definition-cones-suspensions} that
$C(-)$ and $\Sigma(-)$ denote the cone and suspension 
of chain complexes.

\begin{definition}\label{definition-translate-quotients}
	Define maps
	\begin{gather*}
		\Phi\colon
		C(\C'(n-1))\longrightarrow F'_0
		\\
		\Psi\colon 
		R\S_{n-1}\otimes_{R\S_{n-p-1}}\Sigma^{p+1}\C'(n-p-1)
		\longrightarrow 
		F'_p/F'_{p-1}
	\end{gather*}
	by
	\begin{gather*}
		\Phi_r((\sigma\otimes 1),0)
		=
		(\sigma s_{n,n-r-1}\otimes 1)
		\\
		\Phi_r(0,(\tau\otimes 1))
		=
		(\tau\otimes 1)
		\\
		\Psi_{(p+1)+r}(
		\sigma\otimes(1\otimes 1))
		=\sigma s_{n,n-p}\otimes 1
	\end{gather*}
	for $\sigma,\tau\in R\S_{n-1}$.
\end{definition}

\begin{proposition}\label{proposition-translate-quotients}
	Under the isomorphism $\Theta$, the isomorphisms
	\[C(\C(n-1))\xrightarrow{\cong} F_0
	\quad\text{and}\quad
	R\S_{n-1}\otimes_{R\S_{n-p-1}}\Sigma^{p+1}\C(n-p-1)
	\xrightarrow{\cong}F_p/F_{p-1}\]
	of Lemmas~\ref{lemma-fzero} and~\ref{lemma-fp}
	correspond to the isomorphisms
	$\Phi$ and $\Psi$ respectively.
\end{proposition}

\begin{proof}
	First we recall that the isomorphism 
	$C(\C(n-1))\xrightarrow{\cong} F_0$
	of Lemma~\ref{lemma-fzero} is 
	given in degree $r$ by
	\[
		((x_0,\ldots,x_r),0)
		\longmapsto
		(x_0,\ldots,x_r),
		\qquad
		(0,(y_0,\ldots,y_{r-1}))
		\longmapsto
		(y_0,\ldots,y_{r-1},n).
	\]
	We must show that under this map we have
	\[
		(\Theta_r(\sigma\otimes 1),0)
		\longmapsto
		\Theta_r\Phi_r(\sigma\otimes 1,0),
		\qquad
		(0,\Theta_r(\tau\otimes 1))
		\longmapsto
		\Theta_r\Phi_r(0,\tau\otimes 1),
	\]
	for $\sigma, \tau\in R\S_{n-1}$.
	We check this by computing directly:
	\begin{align*}
		(\Theta_r(\sigma\otimes 1),0)
		&=
		((\sigma(n-r-1),\ldots,\sigma(n-1)),0)
		\\
		&\longmapsto
		(\sigma(n-r-1),\ldots,\sigma(n-1))
		\\
		&=
		(\sigma s_{n,n-r-1}(n-r),\ldots,\sigma s_{n,n-r-1}(n))
		\\
		&=
		\Theta_r(\sigma s_{n,n-r-1}\otimes 1)
		\\
		&=
		\Theta_r(\Phi_r(\sigma\otimes 1,0))
	\end{align*}
	and
	\begin{align*}
		(0,\Theta_{r-1}(\tau\otimes 1))
		&=
		(0,(\tau(n-r),\ldots,\tau(n-1)))
		\\
		&\longmapsto
		(\tau(n-r),\ldots,\tau(n-1),n)
		\\
		&=
		\Theta_r(\tau\otimes 1)
		\\
		&=
		\Theta_r(\Phi_r(0,\tau\otimes 1)).
	\end{align*}

	Next, the isomorphism 
	$
		R\S_{n-1}\otimes_{R\S_{n-p-1}}\Sigma^{p+1}\C(n-p-1)
		\xrightarrow{\cong}
		F_p/F_{p-1}
	$
	of Lemma~\ref{lemma-fp} is
	given in degree $(p+1)+r$,
	for $-1\leqslant r\leqslant n-p-2$, by 
	\[
		\sigma\otimes (x_0,\ldots,x_r)
		\longmapsto
		\sigma(x_0,\ldots,x_r,n,n-p,\ldots,n-1).
	\]
	Note that
	in this degree $R\S_{n-1}\otimes_{R\S_{n-p-1}}\Sigma^{p+1}\C'(n-p-1)$
	is
	\[
		R\S_{n-1}\otimes_{R\S_{n-p-1}} 
		(R\S_{n-p-1}\otimes_{R\S_{n-p-1-r}}\t)
		\cong
		R\S_{n-1} \otimes_{R\S_{n-p-1-r}}\t
	\]
	so that a typical generator 
	may be written in the form $\sigma\otimes(1\otimes 1)$
	for $\sigma\in\S_{n-1}$.
	We must show that, for such an element, we have
	\[
		\sigma\otimes\Theta_r(1\otimes 1)
		\longmapsto
		\Theta_{(p+1)+r}(\sigma\otimes 1)
	\]
	Now we compute:
	\begin{align*}
		\sigma\otimes \Theta_r(1\otimes 1)
		&=
		\sigma\otimes (n-p-1-r,\ldots,n-p-1)
		\\
		&\longmapsto
		\sigma(n-p-1-r,\ldots,n-p-1,n,n-p,\ldots,n-1)
		\\
		&=\sigma s_{n,n-p}(n-p-1-r,\ldots,n)
		\\
		&=\Theta_{(p+1)+r}(\sigma\otimes 1)\qedhere
	\end{align*}
\end{proof}

\section{An overview of $\D(n)$}
\label{section-Dn-overview}

At the end of the last section we summarised how to replace the complex
of injective words, its filtration, and the identification
of the filtration quotients with a version
\begin{gather*}
	\C'(n)
	\\
	F'_0\subseteq\cdots\subseteq F'_{n-1}=\C'(n)
	\\
	\Phi\colon C(\C'(n-1))\xrightarrow{\ \cong\ } F'_0
	\\
	\Psi\colon R\S_{n-1}\otimes_{R\S_{n-p-1}}\Sigma^{p+1}\C'(n-p-1)
	\xrightarrow{\ \cong\ }
	F'_p/F'_{p-1}
\end{gather*}
that was phrased entirely in terms of the group ring of $\S_n$,
where $\S_n$ is regarded as a Coxeter group.
This functions as a framework for adapting the entire situation
to the setting of Iwahori-Hecke algebras, by 
making the substitutions:
\[\begin{array}{ccc}
	R\S_m & \rightsquigarrow & \H_m
	\\
	s_i & \rightsquigarrow & T_i
	\\
	s_is_js_i = s_js_is_j & \rightsquigarrow & T_iT_jT_i=T_jT_iT_j
	\\
	s_i^2 = 1 & \rightsquigarrow & T_i^2 = (q-1)T_i + q
\end{array}\]
This is indeed the framework we will follow, and we shall 
follow it quite closely, though various obstacles will arise:
\begin{itemize}
	\item
	We will need to find entirely new proofs,
	within the Iwahori-Hecke setting, of the following:
	That the $\partial^r_j$ are well-defined, 
	that the differentials compose to zero,
	that the differentials respect the filtration, 
	that the filtration is exhaustive,
	that the maps $\Phi$ and $\Psi$ are well-defined chain maps, 
	and that they are isomorphisms.
	\item
	The group ring $R\S_m$ is obtained from $\H_m$ in the case
	$q=1$.
	This means that we must expect any and all quantities, once
	extended to the Iwahori-Hecke setting, to possibly
	contain hitherto-invisible powers of $q$.
	\item
	The relation $s_i^2=1$ 
	is deformed to $T_i^2 = (q-1)T_i + q$,
	and this can introduce new difficulties that are not present
	in the case $q=1$.
\end{itemize}
The plan is as follows:
\begin{itemize}
	\item
	In section~\ref{section-Dn} we define our new chain complex
	$\D(n)$, and verify that its differentials are well-defined
	and compose to $0$.
	\item
	In section~\ref{section-filtration} we define the filtration
	$F_0\subseteq\cdots\subseteq F_{n-1}\subseteq\D(n)$,
	and we prove that it consists of subcomplexes, and exhausts $\D(n)$.
	\item
	In section~\ref{section-filtration-quotients} we study
	the filtration quotients,  proving the following theorem.
\end{itemize}

\begin{theorem}\label{theorem-filtration}
	There are isomorphisms 
	\[
		\Phi\colon C(\D(n-1))\xrightarrow{\ \cong\ } F_0,
		\qquad
		\Psi\colon \H_{n-1}\otimes_{\H_{n-p-1}}\Sigma^{p+1}\D(n-p-1)
		\xrightarrow{\ \cong\ } F_p/F_{p-1}.
	\]
	Here $C$ denotes the cone and $\Sigma$ denotes suspension,
	as in Definition~\ref{definition-cones-suspensions},
	while the $\D(-)$ and $F_p$ will be defined in~\ref{definition-Dn} 
	and~\ref{definition-filtration} below.
\end{theorem}

At each step, it will be evident that the new constructions
in the Iwahori-Hecke setting reduce to those in symmetric group setting
in the case $q=1$.
Altogether, these results will prove the following theorem,
which follows exactly as in the proof of Theorem~\ref{theorem-farmer}.

\begin{theorem}\label{theorem-Dn-acyclicity}
	$H_d(\D(n))=0$ for $d\leqslant n-2$.
\end{theorem}

\section{The chain complex $\D(n)$}\label{section-Dn}

In this section we introduce our new chain complex $\D(n)$,
which will be the Iwahori-Hecke analogue of the
complex of injective words,
or more accurately, of the replacement $\C'(n)$ that we obtained
in section~\ref{section-translation}.

To begin with, we introduce elements $T_{ba}\in\H_n$
that are analogous to the elements $s_{ba}\in R\S_n$
which arose in our discussion of $\C'(n)$ and its filtration,
and we establish the properties
that we will need in what follows.

\begin{definition}[The elements $T_{ba}$]\label{definition-Tab}
	Given $n\geqslant b\geqslant a\geqslant 1$ we define
	\[T_{ba} = T_{b-1}T_{b-2}\cdots T_a.\]
	Note that the indices always decrease from left to right,
	so that when $b=a$ the product is empty and $T_{aa}=1$.
	Note also that $T_{a+1\,a}=T_a$.
	Note that $T_{ba}=T_{s_{ba}}$, and in particular
	that it specialises to the element
	$s_{ba}$ of Definition~\ref{definition-sba}
	in the case $q=1$.
	Where appropriate we will separate the indices of $T_{ba}$
	by a comma for clarity, e.g.~$T_{b,a}=T_{ba}$.
\end{definition}

We now need to establish some properties of the $T_{ba}$ for later use.
In the following proposition we will denote closed intervals of 
natural numbers using the usual notation for subsets of $\mathbb{R}$, 
i.e.~$[a,b]=\{a,a+1,\ldots,b\}\subseteq\mathbb{N}$.

\begin{proposition}\label{T-moves}
	\begin{enumerate}
		\item\label{T-commute}
		If $[a,b]\cap[c,d]=\emptyset$, then
		$T_{dc}T_{ba} = T_{ba}T_{dc}$.

		\item\label{T-concatenate}
		If $b\in[a,c]$ then $T_{ca} = T_{cb}T_{ba}$.
		In particular, $T_{ba}=T_{b-1}T_{b-1,a}=T_{b,a+1}T_a$.

		\item\label{T-multi-swap}
		For $[c,d],[c-1,d-1]\subseteq[a,b]$ we have
		$T_{ba}T_{dc} = T_{d-1,c-1}T_{ba}$.
		In particular, if $k\in(a,b)$ then $T_{ba}T_k=T_{k-1}T_{ba}$.

		\item\label{T-clash}
		For $a\leqslant b<c\leqslant d$, we have
		$T_{db}T_{ca}
		=
		(q-1)\cdot T_{c-1,b}T_{da} 
		+ q \cdot T_{c-1,a}T_{d,b+1}$.
	\end{enumerate}
\end{proposition}

\begin{proof}
	\eqref{T-commute} holds because all letters of $T_{dc}$
	commute with all letters of $T_{ba}$.
	\eqref{T-concatenate} is immediate from the definition.
	Next, the braid relation $T_kT_{k-1}T_k=T_{k-1}T_kT_{k-1}$ can be
	written as $T_{k+1,k-1}T_k = T_{k-1}T_{k+1,k-1}$.
	From this it follows that if $k\in(a,b)$
	then:
	\begin{align*}
		T_{ba}T_k 
		&= (T_{b,k+1}T_{k+1,k-1}T_{k-1,a})T_k
		\\
		&= T_{b,k+1}(T_{k+1,k-1}T_k)T_{k-1,a}
		\\
		&= T_{b,k+1}(T_{k-1}T_{k+1,k-1})T_{k-1,a}
		\\
		&= T_{k-1}(T_{b,k+1}T_{k+1,k-1})T_{k-1,a}
		\\
		&= T_{k-1}T_{ba}
	\end{align*}
	Now \eqref{T-multi-swap} follows by repeated applications of the
	last property:
	\[
		T_{ba}T_{dc}=T_{ba}T_{d-1}\cdots T_c
		=T_{d-2}\cdots T_{c-1}T_{ba}=T_{d-1,c-1}T_{ba}
	\]
	We now prove~\eqref{T-clash}, which boils down to 
	the identity $T_b^2=(q-1)T_b+q$.
	\begin{align*}
		T_{db}T_{ca}
		&=
		T_{db}T_{c,b+1}T_{b+1,a}
		\\
		&=
		T_{c-1,b}T_{db}T_{b+1,a}
		\\
		&=
		T_{c-1,b}(T_{d,b+1}T_b)(T_bT_{ba})
		\\
		&=
		T_{c-1,b}T_{d,b+1}( T_b^2) T_{ba}
		\\
		&=
		T_{c-1,b}T_{d,b+1}\big((q-1)T_b+q\big) T_{ba}
		\\
		&=
		(q-1)\cdot T_{c-1,b}T_{d,b+1}T_bT_{ba}
		+q \cdot T_{c-1,b}T_{d,b+1}T_{ba}
		\\
		&=
		(q-1)\cdot T_{c-1,b}T_{da}
		+q \cdot T_{c-1,b}T_{ba}T_{d,b+1}
		\\
		&=
		(q-1)\cdot T_{c-1,b}T_{da}
		+q \cdot T_{c-1,a}T_{d,b+1}
	\end{align*}
	Here we used~\eqref{T-concatenate} multiple times,
	we used~\eqref{T-multi-swap} on the second line,
	and on the seventh line we used~\eqref{T-commute}.	
\end{proof}

Note that properties (1), (2) and (3) above 
were consequences of the braid relation
$T_kT_{k-1}T_k=T_{k-1}T_kT_{k-1}$,
and so reflect properties of the braid group itself.

We are now in a position to write down our complex $\D(n)$.
It is a direct extension of the complex $\C'(n)$
of Definition~\ref{definition-cdashn}, 
which was isomorphic to $\C(n)$
via the isomorphism $\Theta$ of 
Proposition~\ref{proposition-Theta}.
One can see that the complex
$\D(n)$ below reduces to $\C'(n)$ in the case $q=1$.

\begin{definition}[The complex $\D(n)$]\label{definition-Dn}
	Let $n\geqslant 0$.
	The complex $\D(n)$ is defined to be the chain complex of left
	$\H_n$-modules with
	$\D(n)_r=\H_n\otimes_{\H_{n-r-1}}\t$
	for $r$ in the range $-1\leqslant r \leqslant n-1$,
	and with $\D(n)_r=0$ for $r$ outside that range, so that:
	\[\begin{array}{lcl}
		\D(n)_{n-1}
		&=&
		\H_n\otimes_{\H_{0}}\t
		\\
		&\vdots&
		\\
		\D(n)_{r}
		&=&
		\H_n\otimes_{\H_{n-r-1}}\t
		\\
		&\vdots&
		\\
		\D(n)_{-1}
		&=&
		\H_n\otimes_{\H_{n}}\t
	\end{array}\]
	Observe that $\D(-1)\cong\t$
	and $\D(n-1)\cong\H_n$.
	The differential
	\[
		\partial^r\colon\D(n)_r\longrightarrow\D(n)_{r-1}
	\]
	of $\D(n)$ is defined by 
	\[
		\partial^r = \sum_{j=0}^r (-1)^j q^{-j}\partial^r_j
	\]
	where $\partial^r_j\colon\D(n)_r\to\D(n)_{r-1}$
	is given by
	\[
		\partial^r_j (x\otimes y) = (x\cdot T_{n-r+j,n-r})\otimes y.
	\]
\end{definition}

The following lemma explains that $\D(n)$ really is a chain complex,
i.e.~that $\partial^{r-1}\circ\partial^r=0$,
by explaining how to regard $\D(n)$ as an augmented semi-simplicial 
$\H_n$-module.
As we will see, the powers of $q$ appearing in the definition of
$\partial^r$ compensate for the fact that 
$\partial^{r-1}_i\partial^r_j$ and $\partial^{r-1}_{j-1}\partial^r_i$
are only equal up to powers of $q$.

\begin{lemma}\label{lemma-Dn-chain}
	The maps $\partial^r_j\colon\D(n)_r\to\D(n)_{r-1}$ satisfy the 
	relation
	\[
		\partial^{r-1}_i\partial^r_j
		=
		q\cdot
		\partial^{r-1}_{j-1}\partial^r_i
	\]
	for $0\leqslant i<j\leqslant r$.
	Consequently, the maps $q^{-j}\partial^r_j$ satisfy the relation
	\[
		(q^{-i}\partial^{r-1}_i)\circ (q^{-j}\partial^r_j)
		=
		(q^{-(j-1)}\partial^{r-1}_{j-1})\circ(q^{-i}\partial^r_i)
	\]
	for $0\leqslant i<j\leqslant r$,
	making the $\D(n)_r$ into an augmented semi-simplicial $\H_n$-module
	with face maps $q^{-j}\partial^r_j$,
	and $\D(n)$ into the associated chain complex.
	In particular, we have $\partial^{r-1}\circ\partial^r=0$
	for all $r\geqslant 1$.
\end{lemma}

\begin{proof}
	The second relation is simply a rewriting of the first.
	The lemma's final claim then follows from standard semi-simplicial
	machinery.
	To prove the first relation we compute:
	\begin{align*}
		\partial^{r-1}_i\partial^r_j(x\otimes y)
		&=
		\partial^{r-1}_i(xT_{n-r+j,n-r} \otimes y)
		\\
		&= 
		xT_{n-r+j,n-r}T_{n-r+i+1,n-r+1}\otimes y
		\\
		&= 
		xT_{n-r+i,n-r}T_{n-r+j,n-r}\otimes y
		\\
		&= 
		xT_{n-r+i,n-r}T_{n-r+j,n-r+1}T_{n-r}\otimes y
		\\
		&= 
		xT_{n-r+i,n-r}T_{n-r+j,n-r+1}\otimes T_{n-r}y
		\\
		&= 
		q\cdot xT_{n-r+i,n-r}T_{n-r+j,n-r+1}\otimes y
		\\
		&= 
		q\cdot \partial^{r-1}_{j-1}(xT_{n-r+i,n-r}\otimes y)
		\\
		&= 
		q\cdot \partial^{r-1}_{j-1}\partial^r_i(x\otimes y)
	\end{align*}
	Here the third and fourth equalities used, respectively,
	parts~\eqref{T-multi-swap} and~\eqref{T-concatenate}
	of Proposition~\ref{T-moves}.
	The fifth equality comes from the fact that the elements
	lie in $\D(n)_{r-2}=\H_n\otimes_{\H_{n-r+1}}\t$,
	and $\H_{n-r+1}$ contains the element $T_{n-r}$.
\end{proof}

The lemma above demonstrates that the powers of $q$ appearing in
the formula 
$\partial^r=\sum_{j=0}^r (-1)^jq^j\partial^r_j$ are there precisely to
counter the $q$ that appears in the formula 
$\partial^{r-1}_i\partial^r_j = q\cdot \partial^{r-1}_{j-1}\partial^r_i$.
The lemma also explains that this single instance of $q$ arises because the
$T_i$ act on $\t$ as multiplication by $q$.

\begin{remark}
	In Remark~\ref{remark-similarity} we explained that the definition of 
	$\D(n)$ is similar to several that already appear in the literature.
	The same is true of Lemma~\ref{lemma-Dn-chain},
	and the same properties of the braid group that are relevant there
	also underpin the fact that the objects constructed in
	Definition~34 of~\cite{HepworthCoxeter},
	Definition~7.5 of~\cite{Krannich} and Definition~7.1 of~\cite{Boyd}
	are semi-simplicial.
	See in particular Remark~2.8 of~\cite{Krannich} and the paragraph
	that precedes it, and Lemma~7.3 of~\cite{Boyd}.
\end{remark}

\begin{lemma}\label{lemma-Dn-well-defined}
	The maps $\partial^r_i$ are well-defined.
\end{lemma}

\begin{proof}
	We must show that if $\lambda\in\H_{n-r-1}$, then
	$\partial^r_i(x\lambda\otimes y)$ and $\partial^r_i(x\otimes\lambda y)$
	agree, or in other words that $(x\lambda T_{n-r+i,n-r})\otimes y
	= (x T_{n-r+i,n-r})\otimes (\lambda y)$.
	This amounts to
	showing that $T_{n-r+i,n-r}$ commutes 
	with the generators of $\H_{n-r-1}$.
	To see this, observe that $T_{n-r+i,n-r}$ is a word in 
	$T_{n-r},\ldots,T_{n-r+i-1}$,
	while $\H_{n-r-1}$ is generated by $T_1,\ldots,T_{n-r-2}$,
	and each of the former commutes with each of the latter.
\end{proof}

We conclude the section with the following lemma,
which examines the effect of the maps $\partial^r_j$ on certain specific
elements of $\D(n)$.
It will be used in the sections that follow.

\begin{lemma}\label{lemma-four-cases}
	Suppose that 
	$0\leqslant r\leqslant n-1$, $0\leqslant j\leqslant r$,
	and $0\leqslant t\leqslant r+1$.
	Then we have:
	\[
		\partial^r_j(T_{n,n-t}\otimes 1)
		=
		\begin{cases}
			T_{n-r+j,n-r}(T_{n,n-t}\otimes 1)
			&
			j\leqslant r-t-1
			\\
			\\
			T_{n,n-r}\otimes 1
			&
			j=r-t
			\\
			\\
			(q-1) T_{n-r+j-1,n-t}(T_{n,n-(r-1)-1}\otimes 1)
			\\
			\qquad + q T_{n-r+j-1,n-r}(T_{n,n-(t-1)} \otimes 1)
			&
			j\geqslant r-t+1, \ r\geqslant t
			\\
			\\
			qT_{n-r+j-1,n-r-1}(T_{n,n-r}\otimes 1)
			&
			t=r+1
		\end{cases}
	\]
	In the first case, the resulting quantity is 
	equal to $qT_{n-r+j,n-r}(T_{n,n-(t-1)}\otimes 1)$.
	Note that the first three cases exclude the possibility that
	$t=r+1$, while the final case holds for $t=r+1$ regardless of the 
	value of $j$ in the range $0\leqslant j\leqslant r$.
\end{lemma}

\begin{proof}
	We consider the three cases in turn.

	\emph{Case 1:} $j\leqslant r-t-1$.
	Here we have 
	\begin{align*}
		\partial^r_j(T_{n,n-t}\otimes 1)
		&=
		T_{n,n-t}T_{n-r+j,n-r}\otimes 1
		\\
		&=
		T_{n-r+j,n-r}(T_{n,n-t}\otimes 1)
	\end{align*}
	where in the second equality we used the upper bound on $j$ 
	to show that $n-r+j\leqslant n-t-1$,
	so that $T_{n,n-t}$ and $T_{n-r+j,n-r}$ commute.
	For the final sentence of the lemma, we have	
	\begin{align*}
		T_{n-r+j,n-r}(T_{n,n-t}\otimes 1)
		&=
		T_{n-r+j,n-r}(T_{n,n-t+1}T_{n-t}\otimes 1)
		\\
		&=
		T_{n-r+j,n-r}(T_{n,n-t+1}\otimes T_{n-t}\cdot 1)
		\\
		&=
		qT_{n-r+j,n-r}(T_{n,n-(t-1)}\otimes 1)
	\end{align*}
	where in the second equality we used the
	fact that $t\leqslant r$ to conclude that
	$T_{n-t}$ lies in $\H_{n-r}$ and can therefore be moved past the
	tensor product
	(we are working in $\D(n)_{r-1} = \H_n\otimes_{\H_{n-r}}\t$).
	
	\emph{Case 2:} $j=r-t$.
	Here we have
	\[
		\partial^r_j(T_{n,n-t}\otimes 1)
		=
		T_{n,n-t}T_{n-r+j,n-r}\otimes 1
		=
		T_{n,n-t}T_{n-t,n-r}\otimes 1
		=
		T_{n,n-r}\otimes 1
	\]
	as claimed.

	\emph{Case 3:} $j\geqslant r-t+1$, $r\geqslant t$.
	Here we have
	$n-r\leqslant n-t<n-r+j\leqslant n$ so that by
	part \eqref{T-clash} of Proposition~\ref{T-moves}	
	we have:
	\begin{align*}
		\partial^r_j(T_{n,n-t}\otimes 1)
		&=
		T_{n,n-t}T_{n-r+j,n-r}\otimes 1
		\\
		&=
		(q-1) T_{n-r+j-1,n-t}T_{n,n-r}\otimes 1
		\\
		&\qquad + q T_{n-r+j-1,n-r}T_{n,n-t+1} \otimes 1
		\\
		&=
		(q-1) T_{n-r+j-1,n-t}(T_{n,n-(r-1)-1}\otimes 1)
		\\
		&\qquad + q T_{n-r+j-1,n-r}(T_{n,n-(t-1)} \otimes 1)
	\end{align*}
	as required.

	\emph{Case 4:} $t=r+1$.
	Then by part~\eqref{T-multi-swap} of Proposition~\ref{T-moves} we have:
	\begin{align*}
		\partial^r_j(T_{n,n-r-1}\otimes 1)
		&=T_{n,n-r-1}T_{n-r+j,n-r}\otimes 1
		\\
		&=T_{n-r+j-1,n-r-1}T_{n,n-r-1}\otimes 1
		\\
		&=T_{n-r+j-1,n-r-1}T_{n,n-r}T_{n-r-1}\otimes 1
		\\
		&=T_{n-r+j-1,n-r-1}T_{n,n-r}\otimes T_{n-r-1}\cdot 1
		\\
		&=qT_{n-r+j-1,n-r-1}(T_{n,n-r}\otimes 1)
	\end{align*}
	(Again, we are working in $\D(n)_{r-1}=\H_n\otimes_{\H_{n-r}}\t$.)
\end{proof}

\section{The filtration of $\D(n)$}\label{section-filtration}

In the last section we constructed our complex $\D(n)$,
and now we will generalise the filtration
of the complex of injective words to $\D(n)$.
In Definition~\ref{definition-Fpdash} and 
Proposition~\ref{proposition-Fpdash} we reformulated 
that filtration as 
$F_0'\subseteq\cdots\subseteq F_{n-1}'=\C'(n)$ 
where $F'_p$ is defined, in degree $r$,
to be the $R\S_{n-1}$-span of the elements $s_{n,n-r-1}\otimes 1$
(only in the cases $r=-1,\ldots,n-2$) and the elements
$s_{n,n-t}\otimes 1$ for $t=0,\ldots,\min(r,p)$ 
(only in the cases $r=0,\ldots,n-1$).
We generalise this in the next definition.
Again, in the case $q=1$ we will precisely recover the definition
of the $F'_p$.

\begin{definition}[The filtration $F_p$ of $\D(n)$]\label{definition-filtration}
	Let $0\leqslant p\leqslant (n-1)$.
	Recall that $\D(n)$ in degree $r$ is given by the tensor product
	$\D(n)_r = \H_n\otimes_{\H_{n-r-1}}\t$.
	Define $F_p\subseteq \D(n)$ to be the subcomplex of $\D(n)$
	that in degree $r$ is generated as an $\H_{n-1}$-module
	by all elements of the following two kinds:
	\begin{enumerate}
		\item
		In the cases $r=-1,\ldots, n-2$,
		the element 
		\[
			T_{n,n-r-1}\otimes 1.
		\]
		\item
		In the cases $r= 0,\ldots, n-1$, 
		the elements 
		\[
			T_{n,n-t}\otimes 1
		\]
		for $t=0,\ldots,\min(r,p)$.
	\end{enumerate}
	Thus we obtain a filtration
	\[
		F_0\subseteq F_1\subseteq\cdots \subseteq F_{n-1}\subseteq\D(n)
	\]
	by $\H_{n-1}$-submodules.
	Note that $F_p$ is \emph{not} an $\H_n$-submodule of $\D(n)$.
\end{definition}

Despite the similarity between $T_{n,n-r-1}\otimes 1$
and $T_{n,n-t}\otimes 1$, it will be important to maintain the distinction
between the two kinds of generator of $F_p$, for example to later construct
the isomorphism $F_0\cong C(\D(n-1))$.
See Remark~\ref{remark-moving-letters},
which discussed the role of the words $s_{n,n-r-1}$ and $s_{n,n-t}$
in the definition of $\C'(n)$, 
and Remark~\ref{remark-distinction},
which discussed the importance of maintaining the distinction between
$s_{n,n-r-1}$ and $s_{n,n-t}$.

We must show that the filtration exhausts $\D(n)$,
and that it consists of subcomplexes of $\D(n)$.

\begin{lemma}\label{lemma-exhaustion}
	$F_{n-1}=\D(n)$.
\end{lemma}

\begin{proof}
	In degree $r$, 
	$F_{n-1}$ is the $\H_{n-1}$-span of the elements of the form
	$T_{n,n-t}\otimes 1$ for $t=0,\ldots,r+1$.
	(The case $t=r+1$ is $T_{n,n-r-1}\otimes 1$.)
	We must show that these elements span
	$\D(n)_r=\H_n\otimes_{\H_{n-r-1}}\t$.

	To begin we claim that $\H_n$ is the $\H_{n-1}$-span
	of the elements $T_{nb}$ for $b\in\{1,\ldots,n\}$.
	To see this write $S_m=\{s_1,\ldots,s_{m-1}\}$
	and consider the Coxeter system $(\S_n,S_n)$.
	The subgroup of $\S_n$ generated by $S_{n-1}$ is
	precisely $\S_{n-1}$.
	In Lemma~\ref{lemma-left-minimal}
	we will see that an element of $\S_n$ is $(\S_{n-1},\emptyset)$-reduced
	if and only if it has the form $s_{nb}$ for some $b=1,\ldots,n$.
	Thus any element of $\S_n$ has a representation of the form
	$x\cdot s_{nb}$ for some $b\in\{1,\ldots,n\}$, where $x\in\S_{n-1}$,
	and $\ell(x\cdot s_{nb})=\ell(x)+\ell(s_{nb})$.
	Recall that $\H_n$ has basis given by the elements
	$T_{\sigma}$ for $\sigma\in\S_n$.
	Consequently, $\H_n$ is spanned by elements of the form
	$T_{x\cdot s_{nb}} = T_x\cdot T_{s_{nb}}=T_x\cdot T_{nb}$
	where $x\in\S_{n-1}$ and $1\leqslant b\leqslant n$.
	But in any such product the factor $T_x$ lies in $\H_{n-1}$,
	and this proves our claim.

	Now, in any degree $r$ we have $\D(n)_r=\H_n\otimes_{\H_{n-r-1}}\t$,
	and by the last paragraph this is the $\H_{n-1}$ span 
	of the elements $T_{nb}\otimes 1$ for $b\in\{1,\ldots,n\}$.
	However, note that if $b\leqslant n-r-1$, then
	$T_{nb}\otimes 1 = T_{n,n-r-1}T_{n-r-1,b}\otimes 1
	= T_{n,n-r-1}\otimes T_{n-r-1,b}\cdot 1 
	= q^{n-r-1-b}( T_{n,n-r-1}\otimes 1)$,
	because $T_{n-r-1,b} = T_{n-r-2}\cdots T_b$ lies in $\H_{n-r-1}$.
	It then follows that $\D(n)_r=\H_n\otimes_{\H_{n-r-1}}\t$ is the
	$\H_{n-1}$-span of the elements $T_{n,n-t}\otimes 1$
	for $t=0,\ldots,r+1$, as required.
\end{proof}

\begin{lemma}
	$F_p$ is a subcomplex of $\D(n)$.
\end{lemma}

\begin{proof}
	Fix $0\leqslant p\leqslant (n-1)$ and $0\leqslant r\leqslant (n-1)$.
	Definition~\ref{definition-filtration} lists the generators of $F_p$,
	as an $\H_{n-1}$-module, in each degree $r$.
	So it is enough for us to fix $0\leqslant j \leqslant r$
	and then check that if we apply 
	$\partial^r_j$ to one of the generators of $F_p$ in degree $r$,
	then the result is an $\H_{n-1}$-linear combination of the generators
	of $F_p$ in degree $r-1$.
	
	First let us take $0\leqslant r\leqslant (n-2)$ 
	and consider the generator $T_{n,n-r-1}\otimes 1$ of $F_p$
	in degree $r$.
	Then by the final case of Lemma~\ref{lemma-four-cases}
	we have:
	\begin{align*}
		\partial^r_j(T_{n,n-r-1}\otimes 1)
		&=qT_{n-r+j-1,n-r-1}(T_{n,n-r}\otimes 1)
		\\
		&=qT_{(n-1)-r+j,(n-1)-r}(T_{n,n-(r-1)-1}\otimes 1)
	\end{align*}
	Now observe that $T_{n,n-(r-1)-1}\otimes 1$ is one of the generators
	of $F_p$ in degree $r-1$, while $T_{n-r+j-1,n-r-1}\in\H_{n-1}$,
	so that the final result lies in $F_p$ as required.

	Next, let us take $0\leqslant r\leqslant (n-1)$ and 
	$0\leqslant t\leqslant\min(r,p)$, 
	and consider the generator $T_{n,n-t}\otimes 1$ of $F_p$
	in degree $r$.
	We must check that each 
	$\partial^r_j(T_{n,n-t}\otimes 1)$
	lies in $F_p$ in degree $r-1$.
	We split this into three cases depending on the value of $j$, 
	and in each case we apply the relevant part of 	
	Lemma~\ref{lemma-four-cases}.

	\emph{Case 1:} $j\leqslant r-t-1$.
	Here we have 
	\[
		\partial^r_j(T_{n,n-t}\otimes 1)
		=
		qT_{n-r+j,n-r}(T_{n,n-(t-1)}\otimes 1).
	\]
	The upper bound on $j$ shows that $T_{n-r+j,n-r}\in\H_{n-1}$,
	while $T_{n,n-(t-1)}$ is a generator of $F_p$ in degree $r-1$,
	so that we can conclude that the above expression lies in $F_p$.

	\emph{Case 2:} $j=r-t$.
	Here we have
	\[
		\partial^r_j(T_{n,n-t}\otimes 1)
		=
		T_{n,n-r}\otimes 1
		=
		T_{n,n-(r-1)-1}\otimes 1
	\]
	which is one of the generators of $F_p$
	in degree $r-1$.

	\emph{Case 3:} $j\geqslant r-t+1$.
	Here we have
	\begin{align*}
		\partial^r_j(T_{n,n-t}\otimes 1)
		&=
		(q-1) T_{n-r+j-1,n-t}(T_{n,n-(r-1)-1}\otimes 1)
		\\
		&\qquad + q T_{n-r+j-1,n-r}(T_{n,n-(t-1)} \otimes 1)
	\end{align*}
	Observe that $T_{n,n-(r-1)-1}\otimes 1$ and $T_{n,n-(t-1)} \otimes 1$
	are generators of $F_p$ in degree $r-1$, 
	while $T_{n-r+j-1,n-t}$ and $T_{n-r+j-1,n-r}$ both lie in
	$\H_{n-1}$ since $j\leqslant r$, so that the above expression lies
	in $F_p$, as required.
\end{proof}

\section{Identifying the filtration quotients}
\label{section-filtration-quotients}
By now we have defined our Iwahori-Hecke analogue $\D(n)$ of the complex
of injective words, together with the corresponding filtration 
$F_0\subseteq \cdots \subseteq F_{n-1}\subseteq\D(n)$.
The next step is to understand the filtration quotients
$F_0=F_0/F_{-1}$ and $F_p/F_{p-1}$ for $p\geqslant 1$,
and prove Theorem~\ref{theorem-filtration}.

In Definition~\ref{definition-translate-quotients}
and Proposition~\ref{proposition-translate-quotients}
we reformulated the identification of the filtration quotients of
the complex of injective words
as the isomorphisms 
\[
	\Phi\colon C(\C'(n-1))\xrightarrow{\ \cong\ } F'_0,
	\qquad
	\Psi\colon R\S_{n-1}\otimes_{R\S_{n-p-1}}\Sigma^{p+1}\C'(n-p-1)
	\xrightarrow{\ \cong\ } F_p'/F_{p-1}'
\]
defined by 
\begin{gather*}
	\Phi_r((\sigma\otimes 1),0)
	=
	(\sigma s_{n,n-r-1}\otimes 1)
	\\
	\Phi_r(0,(\tau\otimes 1))
	=
	(\tau\otimes 1)
	\\
	\Psi_{(p+1)+r}(
	\sigma\otimes(1\otimes 1))
	=\sigma s_{n,n-p}\otimes 1
\end{gather*}
for $\sigma,\tau\in R\S_{n-1}$.
In this section we will extend these isomorphisms to the setting of
Iwahori-Hecke algebras.
In subsections~\ref{subsection-phi} and~\ref{subsection-psi}
we will define our extensions of $\Phi$ and $\Psi$,
and prove that they are well-defined chain maps.
Then in subsection~\ref{subsection-isomorphism} we will prove that the
extended maps are isomorphisms, 
completing the proof of Theorem~\ref{theorem-filtration}.

\subsection{Defining $\Phi$}\label{subsection-phi}

Recall that the map $\Phi\colon C(\C'(n-1))\to F'_0$
is defined by 
$\Phi_r((\sigma\otimes 1),0) = (\sigma s_{n,n-r-1}\otimes 1)$
and
$\Phi_r(0,(\tau\otimes 1)) = (\tau\otimes 1)$
for $\sigma,\tau\in\S_{n-1}$.
In the present subsection we will extend this map to the Iwahori-Hecke setting
and prove that it is a well-defined chain map.

Let us recall from Definition~\ref{definition-cones-suspensions}
that the cone $C(\D(n-1))$ is given in degree $r$ by
\[
	C(\D(n-1))_r
	=
	\D(n-1)_r\oplus \D(n-1)_{r-1}
	=
	(\H_{n-1}\otimes_{\H_{n-r-2}}\t) \oplus 
	(\H_{n-1}\otimes_{\H_{n-r-1}}\t)
\]
and therefore has one or two generators as an $\H_{n-1}$-module,
namely $(1\otimes 1,0)$ (only when $-1\leqslant r\leqslant n-2$)
and $(0,1\otimes 1)$ (only when $0\leqslant r\leqslant n-1)$.
Similarly, $F_0$ in degree $r$ is the 
$\H_{n-1}$-submodule of $\D(n)_r=\H_n\otimes_{\H_{n-r-1}}\t$
with generators $T_{n,n-r-1}\otimes 1$ (only when $-1\leqslant r\leqslant n-2$)
and $T_{n,n}\otimes 1=1\otimes 1$ (only when $0\leqslant r \leqslant n-1$).
Observe that in most degrees, both $C(\D(n-1))$ and $F_0$ have two generators,
but that in the extreme cases $r=-1,n-1$ they each have just one generator.

\begin{definition}
	Define 
	\[
		\Phi\colon C(\D(n-1))\longrightarrow F_0
	\]
	to be the $\H_{n-1}$-linear map defined on generators in degree $r$ by
	\[
		\Phi_r(1\otimes 1,0)
		=
		q^{-r}\cdot T_{n,n-r-1}\otimes 1
	\]
	and
	\[
		\Phi_r(0,1\otimes 1)
		=
		q\cdot 1\otimes 1.
	\]
	In other words, given 
	\[
		(x\otimes 1,y\otimes 1)\in
		C(\D(n-1))_r = 
		(\H_{n-1}\otimes_{\H_{n-r-2}}\t) \oplus 
		(\H_{n-1}\otimes_{\H_{n-r-1}}\t)
	\]
	we have 
	\[
		\Phi_r(x\otimes 1,y\otimes 1)=
		q^{-r} xT_{n,n-r-1}\otimes 1 + q y\otimes 1
		\in \D(n)_r = \H_n\otimes_{\H_{n-r-1}}\t.
	\]
\end{definition}

One sees immediately that this restricts to the original definition
in the case $q=1$.

\begin{lemma}
	$\Phi$ is well defined.
\end{lemma}

\begin{proof}
	For $(x\otimes 1,y\otimes 1)$ as in the definition,
	we must check that 
	$\Phi_r(x\otimes 1,y\otimes 1)$
	depends only on $x\otimes 1$ and $y\otimes 1$, 
	and not on $x$ and $y$ themselves.
	This means that we must check that
	\[
		\Phi_r(xT_k\otimes 1,yT_{k'}\otimes 1)
		=
		\Phi_r(x\otimes (T_k\cdot 1),y\otimes (T_{k'}\cdot1))
	\]
	for $1\leqslant k \leqslant n-r-3$ and
	$1\leqslant k'\leqslant n-r-2$.
	This amounts to showing that
	$xT_kT_{n,n-r-1}\otimes 1 = xT_{n,n-r-1}\otimes (T_k\cdot 1)$
	and 
	$yT_{k'}\otimes 1 = y\otimes (T_{k'}\cdot 1)$
	in $\H_n\otimes_{\H_{n-r-1}}\t$.
	The first of these holds since $T_k$ commutes with $T_{n,n-r-1}$
	and lies in $\H_{n-r-1}$,
	and the second is immediate since $T_{k'}$ lies in $\H_{n-r-1}$.
\end{proof}

\begin{lemma}
	$\Phi$ is a chain map.
\end{lemma}

\begin{proof}
	The map $\Phi$ is by definition $\H_{n-1}$-linear.
	In each degree $r$, $C(\D(n-1))$ has one or two generators
	as an $\H_{n-1}$-module, 
	namely $(1\otimes 1,0)$ (only when $-1\leqslant r\leqslant n-2$)
	and $(0,1\otimes 1)$ (only when $0\leqslant r\leqslant n-1)$.
	So it is sufficient for us to verify that $\Phi$ commutes with
	the differentials when applied to these two generators.
	Recall from Definition~\ref{definition-cones-suspensions}
	that the differential of $C(\D(n-1))$ is given by
	\[
		\partial^r(x\otimes 1,y\otimes 1)
		=
		(\partial^r(x\otimes 1)+(-1)^r (y\otimes 1),
		\partial^{r-1}(y\otimes 1)).
	\]

	First we consider 
	$\partial^r(\Phi_r(1\otimes 1,0))$.
	The final case of Lemma~\ref{lemma-four-cases} gives
	\[
		\partial^r_j(T_{n,n-r-1}\otimes 1)
		=qT_{(n-1)-r+j,(n-1)-r}(T_{n,n-(r-1)-1}\otimes 1)
	\]
	so that we have
	\begin{align*}
		\partial^r(\Phi_r(1\otimes 1,0))
		&=
		\sum_{j=0}^r(-1)^jq^{-j}\partial^r_j\Phi_r(1\otimes 1,0)
		\\
		&=
		\sum_{j=0}^r(-1)^jq^{-j}q^{-r}\partial^r_j(T_{n,n-r-1}\otimes 1)
		\\
		&=
		\sum_{j=0}^r(-1)^jq^{-j}q^{-(r-1)}
		T_{(n-1)-r+j,(n-1)-r}(T_{n,n-(r-1)-1}\otimes 1)
		\\
		&=
		\sum_{j=0}^r(-1)^jq^{-j}
		\Phi_{r-1}(T_{(n-1)-r+j,(n-1)-r}\otimes 1,0)
		\\
		&=
		\sum_{j=0}^r(-1)^jq^{-j}
		\Phi_{r-1}(\partial_j^r(1\otimes 1),0)
		\\
		&=
		\Phi_{r-1}(\partial^r(1\otimes 1),0)
		\\
		&=
		\Phi_{r-1}(\partial^r(1\otimes 1,0))
	\end{align*}
	as required.
	Next we consider $\partial^r(\Phi_r(0,1\otimes 1))$:
	\begin{align*}
		\partial^r(\Phi_r(0,1\otimes 1))
		&=
		\sum_{j=0}^r(-1)^jq^{-j}\partial^r_j\Phi_r(0,1\otimes 1)
		\\
		&=
		\sum_{j=0}^r(-1)^jq^{-j}q\partial^r_j(1\otimes 1)
		\\
		&=
		\sum_{j=0}^r(-1)^jq^{-j}q(T_{n-r+j,n-r}\otimes 1)
	\end{align*}
	The sum of all but the final term can be simplified as:
	\begin{align*}
		\sum_{j=0}^{r-1}(-1)^jq^{-j}q(T_{n-r+j,n-r}\otimes 1)
		&=
		\sum_{j=0}^{r-1}(-1)^jq^{-j}\Phi_{r-1}(0,T_{n-r+j,n-r}\otimes 1)
		\\
		&=
		\sum_{j=0}^{r-1}(-1)^jq^{-j}
		\Phi_{r-1}(0,T_{(n-1)-(r-1)+j,(n-1)-(r-1)}\otimes 1)
		\\
		&=
		\sum_{j=0}^{r-1}
		(-1)^jq^{-j}\Phi_{r-1}(0,\partial^{r-1}_j(1\otimes 1))
		\\
		&=
		\Phi_{r-1}(0,\partial^{r-1}(1\otimes 1))
	\end{align*}
	And the final term can be simplified as:
	\begin{align*}
		(-1)^rq^{-r+1}(T_{n,n-r}\otimes 1)
		&=
		(-1)^rq^{-(r-1)}(T_{n,n-(r-1)-1}\otimes 1)
		\\
		&=
		(-1)^r\Phi_{r-1}(1\otimes 1,0)
	\end{align*}
	So altogether we have 
	\begin{align*}
		\partial^r(\Phi_r(0,1\otimes 1))
		&=
		\Phi_{r-1}(0,\partial^{r-1}(1\otimes 1))
		+
		(-1)^r\Phi_{r-1}(1\otimes 1,0)
		\\
		&=
		\Phi_{r-1}((-1)^r(1\otimes 1),\partial^{r-1}(1\otimes 1))
		\\
		&=
		\Phi_{r-1}(\partial^r(0,1\otimes 1)).
	\end{align*}
	This completes the proof.
\end{proof}

\subsection{Defining $\Psi$}\label{subsection-psi}

Recall that the isomorphism
\[\Psi\colon R\S_{n-1}\otimes_{R\S_{n-p-1}}\Sigma^{p+1}\C'(n-p-1)
\xrightarrow{\ \cong\ } F_p'/F_{p-1}'\]
is defined by 
$\Psi_r(\sigma\otimes(1\otimes 1)) = \sigma s_{n,n-p}\otimes 1$
for $\sigma\in R\S_{n-1}$, or in other words, the map which sends
the generator $1\otimes(1\otimes 1)$ to $s_{n,n-p}\otimes 1$.
In the present section we will extend this map to the Iwahori-Hecke setting,
i.e.~a map
\[
	\Psi\colon \H_{n-1}\otimes_{\H_{n-p-1}}\Sigma^{p+1}\D(n-p-1)
	\to F_p/F_{p-1},
\]
and we will prove that it is a well-defined chain map.
Before defining our map, let us elaborate on its domain and codomain.
We work in a fixed degree $(p+1)+r$ for $-1\leqslant r\leqslant n-p-1$.
\begin{itemize}
	\item
	In degree $(p+1)+r$, the domain is 
	\begin{align*}
		\H_{n-1}&\otimes_{\H_{n-p-1}} 
		\Sigma^{p+1}
		\D(n-p-1)_{(p+1)+r}
		\\
		&=
		\H_{n-1}\otimes_{\H_{n-p-1}} \D(n-p-1)_r
		\\
		&=
		\H_{n-1}\otimes_{\H_{n-p-1}} 
		\left(\H_{n-p-1}\otimes_{\H_{n-p-r-2}}\t\right)
		\\
		&\cong
		\H_{n-1}\otimes_{\H_{n-p-r-2}}\t
	\end{align*}
	and so it is generated as an $\H_{n-1}$-module by
	the element
	$1\otimes(1\otimes 1)$.

	\item
	In degree $(p+1)+r$,
	the codomain $F_p/F_{p-1}$ is a subquotient of
	\[
		\D(n)_{(p+1)+r}
		=
		\H_n\otimes_{\H_{n-p-r-2}}\t,
	\]	
	and it is generated as an $\H_{n-1}$-module by the element
	$T_{n,n-p}\otimes 1$.
	(Recall that $F_p$ in this degree is the
	$\H_{n-1}$ submodule generated by
	$T_{n,n-[(p+1)+r]-1}\otimes 1,
	T_{n,n}\otimes 1,\ldots,T_{n,n-p}\otimes 1$, 
	and that $F_{p-1}$ is the
	$\H_{n-1}$-submodule generated by all these except the last.)
	For brevity we
	will write elements of the quotient $F_p/F_{p-1}$ as
	elements of $F_p$, i.e.~we will not indicate additive cosets or
	equivalence classes.
\end{itemize}

\begin{definition}\label{definition-Psi}
	Let $p\geqslant 1$.
	We define 
	\[
		\Psi\colon
		\H_{n-1}\otimes_{\H_{n-p-1}} \Sigma^{p+1}\D(n-p-1)
		\longrightarrow
		F_p/F_{p-1}
	\]
	to be the $\H_{n-1}$-linear map that in degree $(p+1)+r$
	is defined on the generator by 
	\[
		\Psi_{(p+1)+r}(1\otimes(1\otimes 1))
		=
		T_{n,n-p}\otimes 1.
	\]
\end{definition}

One can see immediately that this specialises to the
original definition in the case $q=1$.

\begin{lemma}
	$\Psi$ is well defined.
\end{lemma}

\begin{proof}
	Taking $p\geqslant 1$ as in the definition of $\Psi$,
	and working in degree $(p+1)+r$ for $r\geqslant -1$,
	the domain of $\Psi_{(p+1)+r}$ is
	\[
		\H_{n-1}\otimes_{\H_{n-p-1}} 
		\left(\H_{n-p-1}\otimes_{\H_{n-p-r-2}}\t\right)
		\cong
		\H_{n-1}\otimes_{\H_{n-p-r-2}}\t
	\]
	and so we must show that for $1\leqslant k\leqslant n-r-p-3$,
	$\Psi_{(p+1)+r}$ sends the elements $T_k\otimes (1\otimes 1)$
	and $1\otimes (1\otimes (T_k\cdot 1))$ to the same element.
	Since
	\[
		\Psi_{(p+1)+r}(T_k \otimes(1\otimes 1))
		=T_k\cdot \Psi_{(p+1)+r}(1 \otimes(1\otimes 1))
		=T_k\cdot (T_{n,n-p}\otimes 1)
		=(T_kT_{n,n-p})\otimes 1
	\]
	and
	\[
		\Psi_{(p+1)+r}(1 \otimes(1\otimes (T_k\cdot 1)))
		=q\cdot \Psi_{(p+1)+r}(1 \otimes(1\otimes 1))
		=q(T_{n,n-p}\otimes 1)
		=T_{n,n-p}\otimes (T_k\cdot 1)
	\]
	it is enough to show that $T_k$ 
	commutes with $T_{n,n-p}$.
	And indeed, since
	$k\leqslant n-r-p-3\leqslant n-p-2$, this follows immediately.
\end{proof}

\begin{lemma}
	$\Psi$ is a chain map.
\end{lemma}

\begin{proof}
	We take $p\geqslant 1$ as in the definition of $\Psi$,
	and we take $r\geqslant -1$.
	Consider the differentials going from degree $r+p+1$ to $r+p$.
	In the domain of $\Psi$ the differential in this degree is
	induced by the boundary map
	\[
		\partial^r=\sum_{j=0}^r(-1)^j q^{-j}\partial^r_j
	\]
	of $\D(n-p-1)$.
	And in the codomain of $\Psi$ the differential in degree $(p+1)+r$ is
	\[
		\partial^{(p+1)+r}
		=\sum_{j=0}^{(p+1)+r}(-1)^j q^{-j}\partial^{(p+1)+r}_j.
	\]
	Now, the codomain is $(F_p/F_{p-1})_{(p+1)+r}$, which is generated
	as an $\H_{n-1}$-module by the single element $T_{n,n-p}\otimes 1$.
	The first three cases of Lemma~\ref{lemma-four-cases} 
	give us the following:
	\[
		\partial^{r+p+1}_j(T_{n,n-p}\otimes 1)
		=
		\begin{cases}
			T_{n-r-p-1+j,n-r-p-1}(T_{n,n-p}\otimes 1)
			&
			j\leqslant r
			\\
			\\
			T_{n,n-(p+r)-1}\otimes 1
			&
			j=r+1
			\\
			\\
			(q-1) T_{n-r-p+j-2,n-p}(T_{n,n-(r+p)-1}\otimes 1)
			\\
			\qquad + q T_{n-r-p+j-2,n-r-p-1}(T_{n,n-(p-1)}\otimes 1)
			&
			j\geqslant r+2
		\end{cases}
	\]
	Both of $T_{n,n-(p+r)-1}\otimes 1$ and $T_{n,n-(p-1)}\otimes 1$ 
	are generators of the $\H_{n-1}$-module $F_{p-1}$ 
	in degree $p+r$, and 
	$T_{n-r-p+j-2,n-p}$ and $T_{n-r-p+j-2,n-r-p-1}$ lie in $\H_{n-1}$.
	So in the second and third cases the results vanish in the quotient
	$F_p/F_{p-1}$, and we obtain:
	\[
		\partial^{r+p+1}_j(T_{n,n-p}\otimes 1)
		=
		\begin{cases}
			T_{n-r-p-1+j,n-r-p-1}(T_{n,n-p}\otimes 1)
			&
			j\leqslant r
			\\
			0
			&
			j\geqslant r+1
		\end{cases}
	\]
	This means that in order to verify that 
	$\Psi_{(p+1)+(r-1)}\circ\partial^{r}
	=\partial^{(p+1)+r}\circ\Psi_{(p+1)+r}$, it will suffice to check that
	$\Psi_{(p+1)+(r-1)}\circ\partial_j^{r} 
	=\partial^{(p+1)+r}_j\circ\Psi_{(p+1)+r}$
	for $0\leqslant j\leqslant r$.
	And since the domain is generated as an $\H_{n-1}$ module by
	the element $1\otimes(1\otimes 1)$, it is enough to verify that
	\[
		\partial^{r+p+1}_j\Psi_{(p+1)+r}(1\otimes(1\otimes 1))
		=
		\Psi_{(p+1)+(r-1)}\partial^r_j(1\otimes(1\otimes 1)).
	\]
	And indeed:
	\begin{align*}
		\partial^{r+p+1}_j\Psi_{(p+1)+r}(1\otimes(1\otimes 1))
		&=
		\partial^{r+p+1}_j(T_{n,n-p}\otimes 1)
		\\
		&=
		T_{n-r-p-1+j,n-r-p-1}(T_{n,n-p}\otimes 1)
		\\
		&=
		\Psi_{(p+1)+(r-1)}(1\otimes (T_{n-r-p-1+j,n-r-p-1}\otimes 1))
		\\
		&=
		\Psi_{(p+1)+(r-1)}(1\otimes (T_{(n-p-1)-r+j,(n-p-1)-r}\otimes 1))
		\\
		&=
		\Psi_{(p+1)+(r-1)}(1\otimes (\partial^r_j(1\otimes 1)))
		\\
		&=
		\Psi_{(p+1)+(r-1)}\partial^r_j(1\otimes (1\otimes 1))
	\end{align*}
	Here the third equality used the fact that $T_{n-r-p-1+j,n-r-p-1}$
	lies in $\H_{n-1}$ and $\Psi$ is $\H_{n-1}$-linear.
\end{proof}

\subsection{$\Phi$ and $\Psi$ are isomorphisms}
\label{subsection-isomorphism}

So far in this section we have defined our 
chain maps
$\Phi\colon C(\D(n-1))\to F_0$ and
$\Psi\colon\H_{n-1}\otimes_{\H_{n-p-1}}\Sigma^{p+1}\D(n-p-1)
\to F_p/F_{p-1}$.
Now we will prove Theorem~\ref{theorem-filtration},
which states that $\Phi$ and $\Psi$ are isomorphisms,
and this 
will conclude the programme set out in section~\ref{section-Dn-overview}.

In order to prove that $\Phi$ and $\Psi$ are isomorphisms,
we will obtain bases for their domains and codomains, and prove
that they induce bijections between these bases.
Now, $\Phi$ and $\Psi$ are maps of $\H_{n-1}$-modules,
whose domains are built out of tensor products of the form
$\H_{n-1}\otimes_{\H_k}\t$, and 
we understand from Proposition~\ref{proposition-basis}
how to give a basis for $\H_{n-1}\otimes_{\H_k}\t$
as an $\H_{n-1}$-module using the 
distinguished coset representatives for $\S_{n-1}/\S_{k}$.
However, the codomains of $\Phi$ and $\Psi$ are built from tensor
products of the form $\H_n\otimes_{\H_k}\t$, and
in order to obtain a basis of this \emph{as an $\H_{n-1}$-module},
we will need to study the distinguished double coset representatives
of $\S_{n-1}\backslash\S_n/\S_k$.
The relevant theory was described in sections~\ref{subsection-cosets}
and~\ref{subsection-iwahori-hecke}.

In what follows we will consider the Coxeter system
$(\S_n,S_n)$ where $S_n=\{s_1,\ldots,s_{n-1}\}$.
We will similarly write $S_k=\{s_1,\ldots,s_{k-1}\}$,
so that the parabolic subgroup of $\S_{n}$ generated by
$S_{k}$ is precisely $\S_k$.
Note that $S_0=S_1=\emptyset$, corresponding to the fact
that $\S_0$ and $\S_1$ are both trivial groups.
We are interested in understanding generators of 
$\D(n)_r = \H_n\otimes_{\H_{n-r-1}}\t$,
which is to say, the distinguished representatives 
$X_{S_{n-r-1}}^{-1}$ for the left
cosets $\S_n/\S_{n-r-1}$.  In particular, in order to study
the filtration $\{F_p\}$ of $\D(n)$, we consider 
$\H_n\otimes_{\H_{n-r-1}}\t$ as an $\H_{n-1}$-module, so that
we will need to compute the distinguished representatives 
$X_{S_{n-1}S_{n-r-1}}$ of the double cosets
$\S_{n-1}\backslash \S_n /\S_{n-r-1}$.

We will need to once again make use of the elements
\[
	s_{ba}=s_{b-1}\cdots s_a,
	\qquad
	1\leqslant a\leqslant b\leqslant n
\]
that we introduced in Definition~\ref{definition-sba},
and which coincide with the $T_{ba}$ in the case $q=1$.
Recall from Remark~\ref{remark-sab-permutation} that
$s_{ba}$ is the cycle $s_{ba}=(b\cdots a)$
which decreases each of $a+1,\ldots,b$ by $1$, and which sends $a$ to $b$.
We will in particular need the elements
\[
	s_{nj} = (n \cdots j)
\]
which move $j$ into position $n$ while preserving the order of the 
remaining letters.

We begin by identifying
the distinguished representatives for $\S_{n-1}\backslash\S_n$.

\begin{lemma}\label{lemma-left-minimal}
	An element of $\S_n$ is $(S_{n-1},\emptyset)$-reduced
	if and only if it has the form $s_{nj}$ for some
	$j$ in the range $1\leqslant j\leqslant n$.
	In other words,
	\[
		X_{S_{n-1}}
		=
		\{s_{nn},\ldots,s_{n1}\}.
	\]
\end{lemma}

\begin{proof}
	First we show that the given elements are all
	$(S_{n-1},\emptyset)$-reduced.  To do so,
	we need only show that they have no reduced expression
	beginning with an element of $S_{n-1}$.
	But the given expressions for the elements clearly admit
	no M-moves, and are therefore reduced,
	and since they do not begin with elements of $S_{n-1}$,
	this makes clear that the elements are $(S_{n-1},\emptyset)$-reduced.

	Now let $w$ be $(\S_{n-1},\emptyset)$-reduced.
	Then either $w=e$ or the first letter of $w$ must be $s_{n-1}$,
	so let
	$n\geqslant j\geqslant 1$ be the smallest element
	such that $w$ has a reduced expression beginning
	$s_{n-1}\cdots s_j$.  (Thus the case $j=n$ corresponds to $w=e$.)
	We will show that $w=s_{n-1}\cdots s_j$.
	Suppose not: then $w$ has a reduced expression beginning
	$s_{n-1}\cdots s_js_i$ for some $i=1,\ldots,(n-1)$.
	We cannot have $i=j$ for then the expression is not reduced.
	We cannot have $i=j-1$ by minimality of $j$.
	We cannot have $i<j-1$ for then $s_{n-1}\cdots s_js_i=
	s_i s_{n-1}\cdots s_j$ and $w$ is not $(S_{n-1},\emptyset)$-reduced.
	And finally we cannot have $i>j$ because then
	$s_{n-1}\cdots s_js_i = s_{i-1}s_{n-1}\cdots s_j$ is again not
	$\S_{n-1}$-reduced on the left.
	So there is no such $i$.
\end{proof}

Now we wish to study the distinguished representatives of the double
cosets $\S_{n-1}\backslash\S_n/\S_{n-r-1}$.

\begin{lemma}\label{lemma-double-coset-reps}
	For $r$ in the range $-1\leqslant r\leqslant (n-1)$,
	a complete set of distinguished double coset representatives
	of $\S_{n-1}\backslash \S_n/\S_{n-r-1}$ 
	is given by:
	\[
		X_{S_{n-1}S_{n-r-1}}
		=
		\left\{\begin{array}{ll}
			\{s_{nn},\ldots,s_{n,n-r-1}\},
			&
			-1\leqslant r\leqslant (n-2)
			\\
			\{s_{nn},\ldots,s_{n1}\},
			& 
			r=(n-1)
		\end{array}\right.
	\]
	Note that the cases $r=n-2$ and $r=n-1$ produce the same value for
	$X_{S_{n-1}S_{n-r-1}}$.  This is due to the fact that 
	$S_1$ and $S_0$ are both empty, corresponding to the fact that
	$\S_1$ and $\S_0$ are both the trivial group.
\end{lemma}

\begin{proof}
	As in the proof of Lemma~\ref{lemma-left-minimal},
	the given expressions admit no $M$-moves.
	Since none of them begin with a generator
	of $\S_{n-1}$ or end with a generator of $\S_{n-r-1}$,
	they are $(S_{n-1},S_{n-r-1})$-reduced.
	They are therefore minimal double coset representatives.
	(Proposition~2.1.7 of~\cite{GeckPfeiffer}.)

	It remains to show that they are a complete set of minimal
	double coset representatives.
	But a minimal double coset representative is 
	$(S_{n-1},\emptyset)$-reduced,
	so by Lemma~\ref{lemma-left-minimal} it
	has the form $s_{n,j}$ for some $j$.
	And for $s_{n,j}$ to be $(\emptyset,S_{n-r-1})$-reduced, we must have
	$j\geqslant n-r-1$, so that $s_{nj}$ is one of the given elements.
\end{proof}

Now we will apply the Mackey formula for left-cosets
in order to understand distinguished representatives of $\S_n/\S_{n-r-1}$
in terms of $\S_{n-1}\backslash\S_n/\S_{n-r-1}$:

\begin{lemma}\label{lemma-mackey}
	Let $-1\leqslant r\leqslant (n-1)$.
	Then:
	\[
		(X_{S_{n-r-1}})^{-1}
		=
		\underbrace{(X_{S_{n-r-1}}^{S_{n-1}})^{-1}s_{n,n}
		\sqcup\cdots\sqcup
		(X_{S_{n-r-1}}^{S_{n-1}})^{-1}s_{n,n-r}}_{r+1\text{ terms}}
		\sqcup
		\underbrace{(X^{S_{n-1}}_{S_{n-r-2}})^{-1}s_{n,n-r-1}}
	\]
	Here the initial union of $r+1$ terms is empty in the case $r=-1$,
	while the final term is omitted in the case $r=(n-1)$.
	Note the difference between 
	$X^{S_{n-1}}_{S_{n-r-1}}$ in the first $r+1$ terms
	and 
	$X^{S_{n-1}}_{S_{n-r-2}}$ in the final term.
\end{lemma}

\begin{proof}
	We will use the Mackey decomposition of section~\ref{subsection-cosets},
	taking
	$S=S_n$, $J=S_{n-r-1}$ and $K=S_{n-1}$, so that we obtain
	\begin{equation}\label{equation-mackey-proof}
		X_{S_{n-r-1}}^{-1} 
		= \bigsqcup_{d\in X_{S_{n-1}S_{n-r-1}}}
		(X^{S_{n-1}}_{S_{n-1}
		\cap 
		\prescript{d}{}\!S_{n-r-1}})^{-1}\cdot d
	\end{equation}
	Recall from section~\ref{subsection-cosets}
	that in the expression
	$X^{S_{n-1}}_{S_{n-1}\cap \prescript{d}{}\!S_{n-r-1}}$
	the superscript $d$ denotes conjugation, so that 
	\[
		\prescript{d}{}\!S_{n-r-1} 
		= \{\prescript{d}{}x\mid x\in S_{n-r-1}\}
		= \{dxd^{-1}\mid x\in S_{n-r-1}\}.
	\]
	So we must work out $S_{n-1}\cap\prescript{d}{}\!S_{n-r-1}$ for
	$d\in X_{S_{n-1}S_{n-r-1}}$.
	For $0\leqslant r\leqslant (n-1)$ and $0\leqslant k\leqslant r$ 
	we have
	\begin{align*}
		S_{n-1}\cap\prescript{s_{n,n-k}}{}\!S_{n-r-1}
		&=
		S_{n-1}\cap\prescript{s_{n-1}\cdots s_{n-k}}{}
		\!\{s_1,\ldots, s_{n-r-2}\}
		\\
		&=
		S_{n-1}\cap\{s_1,\ldots, s_{n-r-2}\}
		\\
		&=
		S_{n-1}\cap S_{n-r-1}
		\\
		&=
		S_{n-r-1}
	\end{align*}
	since $s_{n-1}\cdots s_{n-k}$ commutes with $s_{1},\ldots,s_{n-r-2}$.
	And for $-1\leqslant r\leqslant (n-2)$ we have
	\begin{align*}
		S_{n-1}\cap\prescript{s_{n,n-r-1}}{}\!S_{n-r-1}
		&=
		S_{n-1}\cap\prescript{s_{n-1}\cdots s_{n-r-1}}{}
		\!\{s_1,\ldots, s_{n-r-2}\}
		\\
		&=
		S_{n-1}\cap\{s_1,\ldots, s_{n-r-3}, 
		s_{n-1}\cdots s_{n-r-2}\cdots s_{n-1}\}
		\\
		&=
		\{s_1,\ldots, s_{n-r-3}\}
		\\
		&=
		S_{n-r-2}
	\end{align*}
	since in this case $s_{n-1}\cdots s_{n-r-1}$ commutes with
	$s_1,\ldots,s_{n-r-3}$, but not with $s_{n-r-2}$.
	The Mackey decomposition~\eqref{equation-mackey-proof}
	now gives us the required result.
\end{proof}

\begin{lemma}\label{lemma-bases}
	Let $-1\leqslant r\leqslant (n-1)$.
	Then $(F_0)_r$ has basis 
	\[
		\{T_xT_{n,n}\otimes 1
		\mid x\in(X_{S_{n-r-1}}^{S_{n-1}})^{-1}\}
		\cup
		\{T_xT_{n,n-r-1}\otimes 1
		\mid x\in(X^{S_{n-1}}_{S_{n-r-2}})^{-1}\}
	\]
	with the first term omitted for $r=-1$ 
	and the last term omitted for $r=n-1$.
	And for $p\geqslant 1$, $(F_p/F_{p-1})_r$ has basis given by 
	\[
		\{T_xT_{n,n-p}\otimes 1
		\mid x\in (X_{S_{n-r-1}}^{S_{n-1}})^{-1}\}
	\]
	so long as $r\geqslant p$, with the basis being empty when $r<p$.
\end{lemma}

\begin{proof}
$\D(n)_r= \H_n\otimes_{\H_{n-r-1}}\t$ has basis 
$\{T_x\otimes 1\mid x\in (X_{S_{n-r-1}})^{-1}\}$.
By Lemma~\ref{lemma-mackey}, this is equal to
\begin{align*}
	\{T_y\otimes 1\mid y\in	(X_{S_{n-r-1}}^{S_{n-1}})^{-1}s_{n,n}\}
	\cup
	\cdots
	\cup
	\{T_y\otimes 1\mid &y\in (X_{S_{n-r-1}}^{S_{n-1}})^{-1}s_{n,n-r}\}
	\\
	\cup&
	\{T_y\otimes 1\mid y\in(X^{S_{n-1}}_{S_{n-r-2}})^{-1}s_{n,n-r-1}\}
\end{align*}
Here the initial union of $r+1$ terms is empty in the case $r=-1$,
while the final term is omitted in the case $r=(n-1)$.
Observe that if $T_y\otimes 1$ is an element of the first set in the union
above, then $y=xs_{n,n-r-1}$ for some $x\in (X^{S_{n-1}}_{S_{n-r-2}})^{-1}$.
Since $\ell(xs_{n,n-r-1}) = \ell(x)+\ell(s_{n,n-r-1})$ 
as in Theorem~\ref{theorem-distinguished}
we then have $T_y = T_xT_{s_{n,n-r-1}}=T_xT_{n,n-r-1}$.  
Similarly for the other sets in the union, so that the basis is given by
\begin{align*}
	\{T_xT_{n,n}\otimes 1\mid x\in(X_{S_{n-r-1}}^{S_{n-1}})^{-1}\}
	\cup
	\cdots
	\cup
	\{T_x&T_{n,n-r}\otimes 1\mid x\in (X_{S_{n-r-1}}^{S_{n-1}})^{-1}\}
	\\
	\cup
	&\{T_xT_{n,n-r-1}\otimes 1\mid x\in(X^{S_{n-1}}_{S_{n-r-2}})^{-1}\}.
\end{align*}
Here the initial union of $r+1$ terms is empty in the case $r=-1$,
while the final term is omitted in the case $r=(n-1)$.
Comparing with the definition of the filtration, we see that 
$(F_p)_r$ has basis given by the union of the first 
$\min(p,r)+1$ of these sets, together with the last (in the cases
that it is present).
In particular, $(F_0)_r$ has basis 
\[
	\{T_xT_{n,n}\otimes 1\mid x\in(X_{S_{n-r-1}}^{S_{n-1}})^{-1}\}
	\cup
	\{T_xT_{n,n-r-1}\otimes 1\mid x\in(X^{S_{n-1}}_{S_{n-r-2}})^{-1}\}
\]
with the first term omitted for $r=-1$ and the last term omitted for $r=n-1$.
And for $p\geqslant 1$, $(F_p/F_{p-1})_r$ has basis given by 
\[
	\{T_xT_{n,n-p}\otimes 1\mid x\in (X_{S_{n-r-1}}^{S_{n-1}})^{-1}\}
\]
so long as $r\geqslant p$, with the basis being empty when $r<p$.
This completes the proof.
\end{proof}

\begin{proof}[Proof of Theorem~\ref{theorem-filtration}]
\label{proof-filtration}

The last lemma established bases for $F_0$ and $F_p/F_{p-1}$ in each degree.
As we will see below, since $C(\D(n-1))$ and 
$\H_{n-1}\otimes_{\H_{n-p-1}}\Sigma^{p+1}\D(n-p-1)$
are defined using tensor products $\H_n\otimes_{\H_k}\t$,
one can write down bases for these in each degree directly
using Proposition~\ref{proposition-basis}.
Then proof that $\Phi$ and $\Psi$ are isomorphisms 
amounts to verifying that they induce bijections between these bases,
at least up to multiplication by powers of $q$.

The domain of $\Phi$ in degree $r$ is
\begin{align*}
	C(\D(n-1))_r
	&=
	\D(n-1)_r\oplus\D(n-1)_{r-1}
	\\
	&=
	(\H_{n-1}\otimes_{\H_{n-r-2}}\t)
	\oplus
	(\H_{n-1}\otimes_{\H_{n-r-1}}\t)
\end{align*}
and therefore has basis
\[
	\{(0,T_x\otimes 1)\mid x\in(X_{S_{n-r-1}}^{S_{n-1}})^{-1}\}
	\cup
	\{(T_x\otimes 1,0)\mid x\in(X_{S_{n-r-2}}^{S_{n-1}})^{-1}\},
\]
with the first term omitted when $r=-1$ and the second omitted
when $r=(n-1)$.
By the definition of $\Phi$ we have
\[
	\Phi_r(T_x\otimes 1,0)=q^{-r}\cdot T_xT_{n,n-r-1}\otimes 1
\]
and
\[
	\Phi_r(0,T_x\otimes 1)=q\cdot T_x\otimes 1 = q\cdot T_xT_{n,n}\otimes 1
\]
so that, up to scaling by powers of $q$,
$\Phi_r$ restricts to a bijection between the basis of $C(\D(n-1))_r$
and the basis of $(F_0)_r$ given in Lemma~\ref{lemma-bases},
so that $\Phi$ is an isomorphism.

Similarly, the domain of $\Psi$ in degree $(p+1)+r$ is 
\begin{align*}
	\H_{n-1}\otimes_{\H_{n-p-1}}\Sigma^{p+1} \D(n-p-1)_{(p+1)+r}
	&=
	\H_{n-1}\otimes_{\H_{n-p-1}}\D(n-p-1)_r
	\\
	&=
	\H_{n-1}\otimes_{\H_{n-p-1}}(\H_{n-p-1}\otimes_{\H_{n-p-r-2}}\t)
	\\
	&\cong
	\H_{n-1}\otimes_{\H_{n-p-r-2}}\t
\end{align*}
and therefore has basis
\[
	\{T_x\otimes (1\otimes 1)\mid x\in(X_{S_{n-p-r-2}}^{S_{n-1}})^{-1}\}
\]
The definition of $\Psi$ gives
\[
	\Psi_{(p+1)+r}(T_x\otimes (1\otimes 1))=T_xT_{n,n-p}\otimes 1
\]
so that $\Psi$ induces a bijection between 
this basis, and the basis of $(F_p/F_{p-1})_r$ 
given in Lemma~\ref{lemma-bases},
and therefore $\Psi$ is an isomorphism.
\end{proof}

\section{Obtaining the spectral sequence}\label{section-spectral-sequence}

The proof of Theorem~\ref{theorem-main} will consist of an intricate
but well-known spectral sequence argument.
Since this argument is not normally carried out in the context of algebras,
we take some time here to explain how the relevant spectral
sequence is constructed in this context.
The results are summarised in the following proposition,
which we will prove over the course of the section.

\begin{proposition}\label{proposition-spectral-sequence}
	There is a homological spectral sequence $\{E^r\}_{r\geqslant 1}$
	with the following properties:
	\begin{itemize}
		\item
		$E^1_{s,t}$ is concentrated in horizontal degrees
		$s\geqslant -1$.

		\item
		$E^1_{s,t} 
		= \tor_t^{\H_{n-s-1}}(\t,\t)$ 
		
		\item
		$d^1\colon E^1_{s,t}\to E^1_{s-1,t}$ 
		is the stabilisation map when $s$ is even,
		and vanishes when $s$ is odd.
		
		\item
		$E^\infty_{s,t}=0$ in total
		degrees $s+t\leqslant (n-2)$.
	\end{itemize}
	Similarly, there is a cohomological spectral sequence 
	$\{E^r\}_{r\geqslant 1}$ with the following properties:
	\begin{itemize}
		\item
		$E_1^{s,t}$ is concentrated in horizontal degrees
		$s\geqslant -1$.

		\item
		$E_1^{s,t} 
		= \ext^t_{\H_{n-s-1}}(\t,\t)$ 
		
		\item
		$d_1\colon E_1^{s,t}\to E_1^{s-1,t}$ 
		is the stabilisation map when $s$ is even,
		and vanishes when $s$ is odd.
		
		\item
		$E_\infty^{s,t}=0$ in total
		degrees $s+t\leqslant (n-2)$.
	\end{itemize}
\end{proposition}

Throughout the section we will need to fix a projective resolution
$P_\ast$ of $\t$ as a right $\H_n$-module, and an injective resolution
$I^\ast$ of $\t$ as a left $\H_n$-module.
It will be important to note that
by restricting the module structure to any $\H_{n-s-1}$,
$P_\ast$ can be regarded as a projective resolution
of $\t$ as a right $\H_{n-s-1}$-module, 
and $I^\ast$ can be regarded as an injective resolution
$I^\ast$ of $\t$ as a left $\H_{n-s-1}$-module.
This is possible because $\H_n$ 
is free as both a right and left $\H_{n-s-1}$-module (see Proposition~\ref{proposition-basis}),
so that restriction preserves injectivity and projectivity.

We now work towards a proof of the proposition.

\begin{lemma}\label{lemma-spectral-sequence-II}
	There is a homological spectral sequence $\{{}^{II}E^r\}$
	with the following properties:
	\begin{itemize}
		\item
		$\IIE^1_{s,t}$ is concentrated in horizontal degrees
		$s\geqslant -1$.

		\item
		${}^{II}E^1_{s,t} 
		= \tor_t^{\H_n}(\t,\D(n)_s)$
		
		\item
		$d^1\colon {}^{II}E^1_{s,t}\to {}^{II}E^1_{s-1,t}$ 
		is induced by $\partial^s\colon \D(n)_s\to \D(n)_{s-1}$.
		
		\item
		$\IIE^\infty_{s,t}=0$ in total
		degrees $s+t\leqslant (n-2)$.
	\end{itemize}
	Similarly, there is a cohomological spectral sequence $\{{}^{II}E_r\}$
	with the following properties:
	\begin{itemize}
		\item
		$\IIE_1^{s,t}$ is concentrated in horizontal degrees
		$s\geqslant -1$.

		\item
		$\IIE^{s,t}_1=\ext_{\H_n}^t(D(n)_s,\t)$
		
		\item
		$d^1\colon {}^{II}E_1^{s-1,t}\to {}^{II}E_1^{s,t}$ 
		is induced by $\partial^s\colon \D(n)_s\to \D(n)_{s-1}$.
		
		\item
		$\IIE_\infty^{s,t}=0$ in total
		degrees $s+t\leqslant (n-2)$.
	\end{itemize}
\end{lemma}

\begin{proof}
	We prove the homological version first.
	Consider the (homological)
	double complex $P_\ast\otimes_{\H_n}\D(n)_\ast$.
	This double complex gives two spectral sequences, 
	$\{{}^IE^r\}$ and $\{{}^{II}E^r\}$, 
	obtained by filtering the totalization by rows or columns.
	In our case, the first spectral sequence has $E^1$ term 
	\[
		{}^IE^1_{s,t} = H_t(P_s\otimes_{\H_n}\D(n)_\ast)
	\]
	with $d^1\colon {}^IE^1_{s,t}\to {}^IE^1_{s-1,t}$ 
	induced by the differential $P_s\to P_{s-1}$.
	The second spectral sequence has $E^1$ term 
	\[
		{}^{II}E^1_{s,t} 
		= H_t(P_\ast\otimes_{\H_n}\D(n)_s)
		= \tor_t^{\H_n}(\t,\D(n)_s)
	\]
	and differential $d^1\colon {}^{II}E^1_{s,t}\to {}^{II}E^1_{s-1,t}$ 
	induced by $\partial^s\colon \D(n)_s\to \D(n)_{s-1}$. 
	Both spectral sequences converge to the homology 
	of the total complex $\mathrm{Tot}(P_\ast\otimes_{\H_n}\D(n)_\ast)$.  
	See section~5.6 of~\cite{Weibel} for details.

	The $E^1$-term of $\{{}^IE^r\}$ can be identified using the fact
	that $P_s$ is projective, so that the functor
	$(P_s\otimes_{\H_n}-)$ commutes with homology, giving us
	\[
		{}^IE^1_{s,t} 
		= H_t(P_s\otimes_{\H_n}\D(n)_\ast)
		\cong
		P_s\otimes_{\H_n} H_t(\D(n)_\ast).
	\]
	But by Theorem~\ref{theorem-Dn-acyclicity}, the right-hand-side vanishes
	for $t\leqslant (n-2)$.  
	In particular, ${}^IE^1_{\ast,\ast}$ vanishes in total degrees
	$\leqslant (n-2)$.  The same therefore holds for all subsequent
	pages of the spectral sequence, so that 
	$H_\ast(\mathrm{Tot}(P_\ast\otimes_{\H_n}\D(n)_\ast))$
	vanishes in degrees $\ast\leqslant (n-2)$.
	Since $\{{}^{II}E^r_{s,t}\}$ also converges to
	$H_\ast(\mathrm{Tot}(P_\ast\otimes_{\H_n}\D(n)_\ast))$,
	we obtain the conclusion.

	For the second case, we consider instead
	the (cohomological) double complex $\hom_{\H_n}(\D(n)_\ast,I^\ast)$.
	One obtains analogous spectral sequences $\{\IE_r\}$
	and $\{\IIE_r\}$, which are analysed in the same way as before.
	In the analysis of $\{\IE_r\}$ one uses 
	the fact that $I^\ast$ is injective and therefore
	$\hom_{\H_n}(-,I^s)$ commutes with homology to show that
	\[
		\IE_1^{s,t} 
		= H^t(\hom_{\H_n}(\D(n)_\ast,I^s))
		\cong
		\hom_{\H_n}(H_t(\D(n)),I^s).\qedhere
	\]
\end{proof}

Having obtained the spectral sequences $\{\IIE^r\}$ and $\{\IIE_r\}$,
we now proceed to turn them into the ones required by 
Proposition~\ref{proposition-spectral-sequence}.
Recall that 
\begin{gather*}
	\IIE^1_{s,t}
	=\tor_t^{\H_n}(\t,\D(n)_s)
	=\tor_t^{\H_n}(\t,\H_n\otimes_{\H_{n-s-1}}\t),
	\\
	\IIE_1^{s,t}
	=\ext^t_{\H_n}(\D(n)_s,\t)
	=\ext^t_{\H_n}(\H_n\otimes_{\H_{n-s-1}}\t,\t).
\end{gather*}
Recall from Proposition~\ref{proposition-basis}
that $\H_n$ is free as a right $\H_{n-s-1}$-module,
so that in particular $\H_n$ is flat as a right $\H_{n-s-1}$-module,
and there is therefore a change-of-rings isomorphism
\[
	\Xi_\ast\colon
	\tor_t^{\H_{n-s-1}}(\t,\t)
	\xrightarrow{\quad\cong\quad}
	\tor_t^{\H_n}(\t,\H_n\otimes_{\H_{n-s-1}}\t)
	=
	\tor_t^{\H_n}(\t,\D(n)_s)
	=
	E^1_{s,t}
\]
given on the level of chain complexes by the isomorphism
\[
	\Xi\colon
	P_\ast\otimes_{\H_{n-s-1}}\t
	\xrightarrow{\quad\cong\quad}
	P_\ast\otimes_{\H_n}(\H_n\otimes_{\H_{n-s-1}}\t),
	\qquad
	\Xi(p\otimes 1)
	=
	p\otimes (1\otimes 1),
\]
with inverse $\Xi^{-1}(p\otimes (h\otimes 1))= ph\otimes 1$.
And there is a change-of-rings isomorphism
\[
	\Xi^\ast\colon
	\ext^t_{\H_{n-s-1}}(\t,\t)
	\xrightarrow{\quad\cong\quad}
	\ext^t_{\H_n}(\H_n\otimes_{\H_{n-s-1}}\t,\t)
	=
	\ext^t_{\H_n}(\D(n)_s,\t)
	=
	E_1^{s,t}.
\]
given on the level of chain complexes by the isomorphism
\[
	\Xi\colon
	\hom_{\H_{n-s-1}}(\t,I^\ast)
	\xrightarrow{\quad\cong\quad}
	\hom_{\H_n}(\H_n\otimes_{\H_{n-s-1}}\t,I^\ast)
	\qquad
	\Xi(f)(h\otimes 1)
	=
	h\cdot f(1)
\]
with inverse $\Xi^{-1}(g)(1) = g(1\otimes 1)$.

We now define $\{E^r\}$ to be simply the spectral sequence $\{\IIE^r\}$,
but with the $E^1$-term modified by replacing 
$\IIE^r_{s,t}=\tor^{\H_n}_{t}(\t,\D(n)_s)$ with
$\tor^{\H_{n-s-1}}_t(\t,\t)$ using the map $\Xi_\ast$, and then
taking the induced differentials.
And we define $\{E_r\}$ to be $\{\IIE_r\}$ but with $E_1$-term modified
by replacing 
$\IIE_r^{s,t}=\ext_{\H_n}^{t}(\D(n)_s,\t)$ with
$\ext_{\H_{n-s-1}}^t(\t,\t)$ using the map $\Xi^\ast$, and again
taking the induced differentials.
Then $\{E^r\}$ and $\{E_r\}$ have all the properties required by 
Proposition~\ref{proposition-spectral-sequence}, except for the
description of the differentials.

\begin{lemma}\label{lemma-done}
	The composites
	\begin{gather*}
		\Xi_\ast^{-1}\circ d^1\circ\Xi_\ast
		\colon
		\tor_\ast^{\H_{n-s-1}}(\t,\t)
		\longrightarrow
		\tor_\ast^{\H_{n-s}}(\t,\t),
		\\
		{\Xi^\ast}^{-1}\circ d_1\circ\Xi^\ast
		\colon
		\ext^\ast_{\H_{n-s}}(\t,\t)
		\longrightarrow
		\ext^\ast_{\H_{n-s-1}}(\t,\t)
	\end{gather*}
	vanish when $s$ is odd,
	and are given by the relevant stabilisation map
	when $s$ is even.
\end{lemma}

\begin{proof}
	Recall that $d^1$ is induced by the differential of $\D(n)$,
	so that it is given on the level of chains by the map
	\begin{gather*}
		\mathrm{id}\otimes\partial^r
		\colon P_\ast\otimes_{\H_n}\D(n)_s
		\longrightarrow
		P_\ast\otimes_{\H_n}\D(n)_{s-1},
		\\
		p\otimes (h\otimes 1)
		\longmapsto
		\sum_{j=0}^s(-1)^jq^{-j}
		(p\otimes (hT_{n-s+j,n-s}\otimes 1)).
	\end{gather*}
	Thus $\Xi_\ast^{-1}\circ d^1\circ\Xi_\ast$ 
	is given on the level of chains by the composite
	\[
		P_\ast\otimes_{\H_{n-s-1}}\t
		\xrightarrow{\Xi}
		P_\ast\otimes_{\H_n}\D(n)_s
		\xrightarrow{\mathrm{id}\otimes\partial_s}
		P_\ast\otimes_{\H_n}\D(n)_{s-1}
		\xrightarrow{\Xi^{-1}}
		P_\ast\otimes_{\H_{n-s}}\t,
	\]
	whose effect on the element $p\otimes 1$ is	
	\begin{align*}
		p\otimes 1
		&\mapsto
		p\otimes(1\otimes 1)
		\\
		&\mapsto 
		\sum_{j=0}^s(-1)^j q^{-j}(p\otimes (T_{n-s+j,n-s}\otimes 1))
		\\
		&\mapsto 
		\sum_{j=0}^s(-1)^j q^{-j}(pT_{n-s+j,n-s}\otimes 1).
	\end{align*}
	By Lemma~\ref{lemma-right-mult} below, this composite is chain
	homotopic to the map
	\begin{gather*}
		P_\ast\otimes_{\H_{n-s-1}}\t
		\longrightarrow
		P_\ast\otimes_{\H_{n-s}}\t
		\\
		p\otimes 1
		\mapsto
		\sum_{j=0}^s(-1)^j q^{-j}q^j(p\otimes 1)
		=
		\sum_{j=0}^s(-1)^j (p\otimes 1)
		=
		\begin{cases}
			p\otimes 1 & s\text{ even}
			\\
			0 & s\text{ odd}
		\end{cases}
	\end{gather*}
	and the result follows in the homological case.
	In the cohomological case the proof is similar, and we leave
	the details to the reader.
\end{proof}

\begin{lemma}\label{lemma-right-mult}
	The map $P_\ast\otimes_{\H_{n-s-1}}\t\to P_\ast\otimes_{\H_{n-s-1}}\t$,
	$p\otimes 1\mapsto pT_{n-s+j,n-s}\otimes 1$ is chain homotopic
	to the map given by multiplication by $q^j$.
	Consequently, the map
	$P_\ast\otimes_{\H_{n-s-1}}\t\to P_\ast\otimes_{\H_{n-s}}\t$,
	$p\otimes 1\mapsto pT_{n-s+j,n-s}\otimes 1$ is chain homotopic
	to the reduction map $P_\ast\otimes_{\H_{n-s-1}}\t
	\to P_\ast\otimes_{\H_{n-s}}\t$ multiplied by $q^j$.

	Analogously, the map 
	$\hom_{\H_{n-s-1}}(\t,I^\ast)\to \hom_{\H_{n-s-1}}(\t,I^\ast)$
	defined by
	$f\mapsto (1\mapsto T_{n-s+j,n-s}\cdot f(1))$ is chain homotopic
	to the map given by multiplication by $q^j$.
	Consequently, the map
	$\hom_{\H_{n-s}}(\t,I^\ast)\to \hom_{\H_{n-s-1}}(\t,I^\ast)$,
	$f\mapsto (1\mapsto T_{n-s+j,n-s}\cdot f(1))$ is chain homotopic
	to the restriction map $\hom_{\H_{n-s}}(\t,I^\ast)
	\to \hom_{\H_{n-s-1}}(\t,I^\ast)$ multiplied by $q^j$.
\end{lemma}

\begin{proof}
	Let us begin with the homological case.
	The next paragraph will show that 
	right-multiplication by $T_{n-s+j,n-s}$ on $P_\ast$ is a map
	of $\H_{n-s-1}$-modules, and that its effect on homology
	is multiplication by $q^j$.
	Another chain map with the
	same properties is multiplication by $q^j$.
	But since $P_\ast$ is a projective resolution by $\H_{n-s-1}$-modules,
	these two maps are chain homotopic.
	
	To see that right-multiplication on $P_\ast$ 
	by $T_{n-s+j,n-s}$ is a map of $\H_{n-s-1}$-modules,
	observe that 
	$T_{n-s+j,n-s}=T_{n-s+j-1}\cdots T_{n-s}$ commutes with  $\H_{n-s-1}$.
	The effect of the map on homology is
	the map $\t\to\t$ that is again 
	given by right multiplication by $T_{n-s+j,n-s}$,
	and since $T_{n-s+j,n-s}$ is a product of $j$ factors $T_k$,
	this is multiplication by $q^j$.

	The proof in the cohomological case is similar,
	left-multiplication by $T_{n-s+j,n-s}$ on $I^\ast$ is a map
	of $\H_{n-s-1}$-modules given on cohomology by multiplication
	by $q^j$, and since $I^\ast$ an injective resolution by
	$\H_{n-s-1}$-modules, this map is chain homotopic to
	multiplication by $q^j$.
\end{proof}

\section{Proof of Theorem~\ref{theorem-main}}\label{section-argument}

We are now able to prove Theorem~\ref{theorem-main},
in the homological case.
The following argument
is essentially what appears in section~5.2 of~\cite{RandalWilliamsConfig},
or in the proof of Theorem~2 of~\cite{Kerz},
except for changes in indexing and notation.

We prove that $\tor_d^{\H_{n-1}}(\t,\t)\to \tor_d^{\H_{n}}(\t,\t)$
is an isomorphism in degrees $d$ satisfying $2d\leqslant n-1$.
We do this by induction on $n$. 
The cases $n=1$ and $n=2$ only make a statement about degree $d=0$
and therefore hold trivially.

Suppose now that $n\geqslant 3$ and that the induction hypothesis
holds for all smaller values of $n$.
In the spectral sequence $\{E^r\}_{r\geqslant 1}$
of Proposition~\ref{proposition-spectral-sequence} the differential
\[
	d^1\colon E^1_{s,t}\to E^1_{s-1,t}
\]
is the stabilisation map
\[
	\tor^{\H_{n-s-1}}_t(\t,\t)\to
	\tor^{\H_{n-s}}_t(\t,\t)
\]
when $s$ is even, and vanishes when $s$ is odd.
In particular, our aim is to show that the maps 
$d^1\colon E^1_{0,t}\to E^1_{-1,t}$ are isomorphisms for $2t\leqslant n-1$,
or in other words that $E^2_{0,t}=0$ and $E^2_{-1,t}=0$
for $2t\leqslant n-1$.

Now let $u\geqslant 1$ and consider the differential
\[
	d^1\colon E^1_{2u,t}\to E^1_{2u-1,t}.
\]
Since this is the stabilisation map, our induction hypothesis states that
it is an isomorphism for $2t\leqslant n-2u-1$. 
This gives the first property below.  The second property follows
easily from it.
\begin{enumerate}
	\item
	For $r\geqslant 2$,
	$E^r_{\ast,\ast}$ vanishes in bidegrees
	$(2u,t)$ and $(2u-1,t)$ for $u\geqslant 1$, $2t\leqslant n-2u-1$.

	\item
	For $r\geqslant 2$,
	$E^r_{\ast,\ast}$ vanishes in bidegrees
	$(s,t)$ satisfying $2t\leqslant n-s-2$ and $s\geqslant 1$.
\end{enumerate}

We now claim that for $r\geqslant 2$ there are no differentials
$d^r$ affecting terms in bidegrees $(-1,t)$ and $(0,t)$
for $2t\leqslant n-1$.
In the case of bidegrees $(-1,t)$,
observe that a $d^r$ landing there 
must originate in bidegree $(-1+r,t-r+1)$, 
but that $E^r_{-1+r,t-r+1}=0$ by property (2) above.
In the case of bidegrees $(0,t)$ and $r\geqslant 3$, 
the same reasoning applies.
In the case of bidegrees $(0,t)$ and $r=2$, the differential
$d^2$ landing there must originate
in $(2,t-1)$, which is $(2u,t-1)$ for $u=1$, and $E^2_{2u,t-1}=0$
by property (1) above.

It follows that if $2t\leqslant n-1$ then
$E^\infty_{-1,t}=E^2_{-1,t}$ 
and $E^\infty_{0,t}=E^2_{0,t}$.
These terms lie in total degrees $d$ satisfying $d\leqslant (n-2)$ 
(this requires our assumption that $n\geqslant 3$).
But by Proposition~\ref{proposition-spectral-sequence}
we know that $E^{\infty}$ vanishes in these total degrees,
so that these terms vanish, and this completes the proof
in the homological case.

The proof in the cohomological case is entirely similar.

\newcommand{\etalchar}[1]{$^{#1}$}
 \newcommand{\noop}[1]{}

\end{document}